\documentclass{article}%
\usepackage{amsmath}
\usepackage{amsfonts}
\usepackage{amssymb}
\usepackage{graphicx}
\usepackage{xcolor}%
\usepackage{soul}
\usepackage{mathrsfs}
\usepackage{mathscinet}
\usepackage{mathtools}
\usepackage{nameref}
\usepackage[title,titletoc]{appendix}
\usepackage{titlesec}
\titlelabel{\thetitle.\ }
\providecommand{\U}[1]{\protect\rule{.1in}{.1in}}
\newtheorem{theorem}{Theorem}

\newtheorem{definition}[theorem]{Definition}

\newtheorem{lemma}[theorem]{Lemma}

\newtheorem{proposition}[theorem]{Proposition}
\newtheorem{remark}[theorem]{Remark}

\newenvironment{proof}[1][Proof]{\noindent\textbf{#1} }{\ \rule{0.5em}{0.5em}}
\numberwithin{theorem}{section}

\begin{document}
 
\author{}
\title{\scshape Calculus of Variations: A Differential Form Approach }
\date{}
\maketitle

\centerline{\scshape Swarnendu Sil}
\medskip
{\footnotesize
 \centerline{ Section de Math\'{e}matiques}
   \centerline{Station 8, EPFL}
   \centerline{1015 Lausanne, Switzerland}
   \centerline{swarnendu.sil@epfl.ch}
}

\begin{abstract}We study integrals of the form $\int_{\Omega}f\left(  d\omega_1 , \ldots , d\omega_m \right), $ where $m \geq 1$ is a given integer, $1 \leq k_{i} \leq n$ 
are integers  and $\omega_{i}$ is a $(k_{i}-1)$-form 
for all $1 \leq i \leq m$ and $ f:\prod_{i=1}^m \Lambda^{k_i}\left(  \mathbb{R}^{n}\right)
\rightarrow\mathbb{R}$ is a continuous function. We introduce the appropriate notions of convexity, namely vectorial ext. one convexity, vectorial ext. 
quasiconvexity and vectorial ext. polyconvexity. We prove weak lower semicontinuity theorems and weak continuity theorems and conclude with applications to minimization problems. 
These results generalize the corresponding results in both classical vectorial calculus of variations and the calculus of variations for a single differential form.  
\end{abstract}
\smallskip


\section{Introduction}
In this article, we study integrals of the form 
$$\int_{\Omega}f\left(  d\omega_1 , \ldots , d\omega_m \right), $$
where $\Omega \subset \mathbb{R}^{n}$ is open and bounded, $m \geq 1$ is a given integer, $1 \leq k_{i} \leq n$ are integers  and $\omega_{i}$ is a $(k_{i}-1)$-form
for all $1 \leq i \leq m$ and $ f:\prod_{i=1}^m \Lambda^{k_i}\left(  \mathbb{R}^{n}\right)
\rightarrow\mathbb{R}$ is a continuous function. When $m=1,$ this problem reduces to the study of the integrals 
$$\int_{\Omega}f\left(  d\omega \right),$$
which was studied systematically in Bandyopadhyay-Dacorogna-Sil \cite{BandDacSil}. On the other hand, when $k_{i}=1$ for all $1 \leq i \leq m,$ the problem can be 
identified with the study of the integrals
$$ \int_{\Omega}f\left(  \nabla  u \right),$$
when $u:\Omega \subset \mathbb{R}^{n} \rightarrow \mathbb{R}^{m}$ is an $\mathbb{R}^{m}$-valued function, which is the classical problem of the calculus of variations, where $m=1$
is called the scalar case and $m >1$ is called the vectorial case. Thus the study of the integrals $\int_{\Omega}f\left(  d\omega_1 , \ldots , d\omega_m \right) $ unifies the classical 
calculus of variations and the calculus of variations for a single 
differential form under a single framework. 

\par  The convexity properties of $f$ plays a crucial role. Generalizing the notions introduced in Bandyopadhyay-Dacorogna-Sil \cite{BandDacSil}, here we introduce the following terminology: \emph{vectorial ext. one convexity}, 
\emph{vectorial ext. quasiconvexity} and \emph{vectorial ext. polyconvexity}. These notions play analogous roles of the classical notions of rank one convexity, 
quasiconvexity and polyconvexity (see, for example Dacorogna \cite{DCV2}) respectively and reduce to precisely those notions 
in the special case when $k_{i}=1$ for all $1 \leq i \leq m.$
The characterization theorem for vectorially ext. quasiaffine functions, obtained for the first time in Sil \cite{silthesis}, is proved. As a corollary, this gives a new proof of the celebrated characterization 
theorem of Ball \cite{Ballquasiaffine}  for quasiaffine functions in the classical case.\smallskip

\par The necessity and sufficiency of vectorial ext. quasiconvexity 
of the map
$(\xi_{1},\ldots, \xi_{m}) \mapsto f(x, \xi_{1},\ldots, \xi_{m}),$ with usual power-type growth condition on $f$,
for the sequential weak lower semicontinuity of integrals of the form 
$$ \int_{\Omega}f\left( x,  d\omega_1 , \ldots , d\omega_m \right),$$
in the larger space $W^{d,\boldsymbol{p}}(\Omega; \boldsymbol{\Lambda^{k-1}})$ is shown, with an additional assumption on traces if $p_{i}=1$ 
but $k_{i} \neq 1$ for some $1 \leq i \leq m.$ 
. Unlike the classical calculus of variations, in general, $W^{d,\boldsymbol{p}},$ instead of 
$W^{1,\boldsymbol{p}},$ is the relevant space from the point of view of coercivity. A counterexample shows the 
result to be optimal in the sense that the semicontinuity result is false if we allow explicit dependence on $\omega_{i}$s in general. This failure is essentially due to the lack of  
Sobolev inequality in $W^{d,\boldsymbol{p}}.$
\par Equivalence of vectorial ext. quasiaffinity with sequential weak continuity of the integrals 
$$\int_{\Omega}f\left( d\omega_1 , \ldots , d\omega_m \right) ,$$ on $W^{d,\boldsymbol{p}}(\Omega; \boldsymbol{\Lambda^{k-1}})$ is proved. 
Sufficiency part of this result however has essentially been obtained in Robbin-Rogers-Temple \cite{RobbinRogersTempleweakcont}. In the spirit of the 
distributional Jacobian determinant in the classical case, two distinct notions of distributional 
wedge product of exact forms are introduced, one generalizing Brezis-Nguyen \cite{BrezisNguyenjacobian} and the other following Iwaniec 
\cite{Iwaniecnonlinearcommutator}. Distributional weak convergence results for such products are proved.
\par Existence theorems for minimization problems for vectorially ext. quasiconvex and 
vectorially ext. polyconvex functions, with possible explicit $x$-dependence are obtained. A counterexample is given to show that minimizer might not exist in general 
if we allow the integrand to depend explicitly on $\omega_{i}$.

\par This achieved unification also both clarifies and raises a number of interesting points, which merit further study. 
\begin{itemize}
 \item The so-called `divergence structure' and cancellations of the determinants, giving rise to improved integrability and weak continuity, is well-known in the classical 
 calculus of variations. It has been exploited in various contexts, namely 
nonlinear elasticity (beginning with \mbox{Ball\cite{Ballquasiaffine}}), theory of `compensated compactness' (\mbox{Coifman-Lions-Meyer-Semmes \cite{CLMS}}, 
\mbox{DiPerna \cite{DipernaCL}}, \mbox{Murat\cite{Muratsufficientnecessarycompensated}}, \mbox{Tartar\cite{TaCC}}), theory of quasiconformal maps and the associated Beltrami fields 
(\mbox{Iwaniec \cite{IwaniecPharmonic}}, \mbox{Iwaniec-Sbordone \cite{IwaniecSbordoneQuasiharmonic}}), very weak solutions of PDEs (\mbox{Sbordone \cite{SbordoneSurveyveryweak}}) etc.
The unified framework views these ideas as central to the calculus of variations as a whole and puts these ideas in their most general and natural setting - the exterior algebra. 
By isolating and clarifying the fundamental core of these ideas, which already proved to be immensely powerful in myriad contexts, the unification can potentially open doorways to new advances 
in nonlinear analysis, especially in a geometric setting.
  
\item On the other hand, from the unified perspective, our ability to settle minimization problems when the integrand have quite general explicit dependence on the $\omega_{i}$s is 
a feature specific to the classical calculus of variations and 
does not extend beyond it. This failure, however, highlights another very fundamental issue, the so-called `gauge invariance' of the minimization problem. Even when $m=1$ but 
$k >1,$ the integrand and thus the minimization problem for $\displaystyle \int_{\Omega} f(x, d\omega)$ is invariant under translation by the infinite dimensional subspace of closed $(k-1)$-forms with 
vanishing boundary values. The lack of coercivity on $W^{1,p}$, unavailability of Sobolev inequality in  $W^{d,p},$ the 
space on which the functional is coercive and the counterexamples to both the semicontinuity and the existence results when general explicit dependence on $\omega$ is allowed are all   
manifestations of this invariance. Also, the crucial fact which allows us to derive existence of minimizers in $W^{1,p}$ is essentially a 
`gauge fixing procedure' (see lemma \mbox{\ref{lemmaforhodge}}). In the general setting of gauge field theories, Uhlenbeck 
\mbox{\cite{UhlenbeckGaugefixing}} proved a gauge fixing result to study Yang-Mills fields, where the energy functional is convex. A better understanding of the interplay between gauge invariance  issues and the introduced convexity notions will likely serve as a stepping stone 
to generalizations of gauge field theories with non-convex energies. 
\end{itemize}

\par The rest of the article is organized as follows. Section \ref{notations} collects all the notations used throughout the article. Section \ref{notionsofconvexity} introduces the convexity notions, derives some basic properties and 
proves the characterization theorem for vectorially quasiaffine functions. Section \ref{weaklowersemicontinuity} and Section \ref{weakcontinuity} discuss sequential weak lower 
semicontinuity and sequential weak continuity results, respectively. Section \ref{existencetheorems} discusses existence theorems for vectorially ext. quasiconvex and vectorially 
ext. polyconvex integrands. \smallskip

\section{Notations}\label{notations}

We gather here the notations which we use throughout this article. We reserve boldface english or greek letters to denote $m$-tuples of integers, real numbers, 
exterior forms etc as explained below.  
\begin{enumerate}
 \item Let $m,n \geq 1$ be integers.
 \begin{itemize}
 \item $\wedge,$ $\lrcorner\,,$ $\left\langle\  ,\ \right\rangle $ and $\ast$ denote the exterior product, the interior product, the
scalar product and the Hodge star operator , respectively.
  \item $\boldsymbol{k}$ stands for an $m$-tuple of integers, $\boldsymbol{k} = (k_1,\ldots, k_m),$ where $1\leq k_i\leq n$ for all $1 \leq i \leq m ,$ where $m \geq 1$
   is a positive integer.
  We  write $\displaystyle \boldsymbol{\Lambda^k}(\mathbb{R}^{n})$ ( or simply $\displaystyle \boldsymbol{\Lambda^k}$) to 
  denote the Cartesian product $\displaystyle \prod_{i=1}^m \Lambda^{k_i}\left(  \mathbb{R}^{n}\right)$, where $\Lambda^{k_i}\left(  \mathbb{R}^{n}\right)$ 
  denotes the vector space of all alternating $k_{i}$-linear maps
$f:\underbrace{\mathbb{R}^{n}\times\cdots\times\mathbb{R}^{n}}_{k_{i}\text{-times}%
}\rightarrow\mathbb{R}.$ For any integer $r$, we also employ the shorthand $\displaystyle \boldsymbol{\Lambda}^{\boldsymbol{k+r}}$ to stand for the product 
$\displaystyle \prod_{i=1}^m \Lambda^{k_i + r}\left(  \mathbb{R}^{n}\right).$ We denote elements of $\displaystyle \boldsymbol{\Lambda^k}$ by boldface greek letters, except $\boldsymbol{\alpha}$, which we reserve for 
multiindices (see below). For example, we write 
$ \displaystyle \boldsymbol{\xi}  \in \boldsymbol{\Lambda^k}$ to mean $\displaystyle \boldsymbol{\xi} =  (\xi_1, \ldots, \xi_m )$ is an $m$-tuple of exterior forms, with 
$\xi_{i} \in \Lambda^{k_{i}}(\mathbb{R}^{n})$ for all $1 \leq i \leq m .$ We also write 
$\displaystyle \lvert \boldsymbol{\xi} \rvert = \left( \sum_{i=1}^{m} \lvert \xi_i \rvert^2\right)^{\frac{1}{2}}.$ In general, boldface greek letters always mean an $m$-tuple of the 
concerned objects. 
\item If $\mathbf{k}$ is an $m$-tuple as defined above, we reserve the boldface greek letter $\boldsymbol{\alpha}$ for a multiindex, i.e an $m$-tuple of integers 
$(\alpha_1,\ldots, \alpha_m) $ with  
 $ 0 \leq \alpha_i \leq \left[ \frac{n}{k_i}\right] $ for all $1 \leq i \leq m .$ We write 
 $ \displaystyle \lvert \boldsymbol{\alpha} \rvert$ and $\displaystyle \lvert \boldsymbol{k\alpha}\rvert $ for the sums $ \displaystyle \sum_{i=1}^{m} \alpha_i$ and  
 $ \displaystyle\sum_{i=1}^{m} k_i\alpha_i ,$ respectively. 
 \item For any $\mathbf{k}$ and $\boldsymbol{\alpha}$, as defined above, such that $\displaystyle 1 \leq \lvert \boldsymbol{k\alpha}\rvert  \leq n, $ we write 
 $\boldsymbol{\xi^{\alpha}}$ for the wedge product
 $$ \xi_{1}^{\alpha_1} \wedge \ldots \wedge \xi_{m}^{\alpha_m} =  \underbrace{\xi_{1}\wedge\cdots\wedge \xi_{1}}_{\alpha_1 \text{-times}}\wedge 
 \ldots \wedge \underbrace{\xi_{m}\wedge\cdots\wedge \xi_{m}}_{\alpha_{m} \text{-times}} \in \Lambda^{\lvert \boldsymbol{k\alpha}\rvert} (\mathbb{R}^{n}).$$
 Clearly, if $\alpha_{i} = 0$ for some $1 \leq i \leq m,$ $\xi_{i}$ is absent from the product. 
 \item Let $\mathbf{k}$ and $\boldsymbol{\alpha}$ be as defined above. Then for any $ \displaystyle \boldsymbol{\xi}  \in \boldsymbol{\Lambda^k}$ and for 
 any integer $1 \leq s \leq n$, $T_{s}(\boldsymbol{\xi})$ stands for the vector with components $\boldsymbol{\xi^{\alpha}} $, where $\boldsymbol{\alpha}$ 
varies over all possible choices such that $\lvert \boldsymbol{\alpha} \rvert = s,$ as long as there is at least one such non-trivial wedge power. As an example, if $m=3,$ then we immediately see that
\begin{align*}
 T_{1}(\boldsymbol{\xi}) &= \left( \xi_{1}, \xi_{2}, \xi_{3}\right) , \\ 
 T_{2}(\boldsymbol{\xi}) &= \left( \xi_{1}^{2}, \xi_{1}\wedge \xi_{2}, \xi_{1}\wedge \xi_{3}, \xi_{2}^{2}, \xi_{2}\wedge \xi_{3}, \xi_{3}^{2}  \right)  \text{ etc.}
\end{align*}
$N(\boldsymbol{k})$ stands for the largest integer $s$ for which there is at least one such non-trivial wedge power, i.e 
$$N(\boldsymbol{k}) = \max \begin{aligned}[t] \left\lbrace
                                         s \in \mathbb{N}: \exists \boldsymbol{\alpha} \text{ with } \lvert \boldsymbol{\alpha}\right. \rvert=s \text{ such}&\text{ that } 
                                         \boldsymbol{\xi^{\alpha}} \neq 0 
 \\  &\left. \text{ for some } \boldsymbol{\xi} \in \boldsymbol{\Lambda^{k}} \right\rbrace .
                                       \end{aligned}$$ 
$T(\boldsymbol{\xi})$ stands for the vector $T(\boldsymbol{\xi}) = \left(T_{1}(\boldsymbol{\xi}), \ldots, T_{N(\boldsymbol{k})}(\boldsymbol{\xi}) \right) ,$ whose 
number of components is denoted by $\tau(n,\boldsymbol{k}), $ i.e $T(\boldsymbol{\xi}) \in \mathbb{R}^{\tau(n,\boldsymbol{k})}.$
 \end{itemize}
\item Let $\boldsymbol{p} = (p_1,\ldots, p_m)$ where $1\leq p_i\leq \infty$ for all $1 \leq i \leq m .$ Let $\Omega \subset \mathbb{R}^{n}$ be open, bounded and smooth.
Let $\nu = \left( \nu_{1},\ldots, \nu_{n} \right)$ denote the outer normal  on $\partial\Omega,$ identified with the $1$-form 
$\displaystyle \nu = \sum_{i=1}^{n} \nu_{i} e^{i} .$ Note that $\nu$ used as a subscript or superscript still denotes just an index and not the normal. There is 
little chance of confusion since the intended meaning is always clear from the context. 
\begin{itemize}
\item Let $0 \leq k \leq n-1$ be an integer and $ 1 \leq p \leq \infty.$ Then we define the following spaces. 
\begin{align*}
    W^{d,p}(\Omega; \Lambda^{k}) &= \left\lbrace \omega \in L^{p}(\Omega; \Lambda^{k}), d\omega \in L^{p}(\Omega; \Lambda^{k+1}) \right\rbrace, \\
    W_{T}^{d,p}(\Omega; \Lambda^{k}) &= \left\lbrace \omega \in L^{p}(\Omega; \Lambda^{k}), d\omega \in L^{p}(\Omega; \Lambda^{k+1}), \nu\wedge \omega = 0 
    \text{ on } \partial\Omega \right\rbrace, \\
    W_{N}^{d,p}(\Omega; \Lambda^{k}) &= \left\lbrace \omega \in L^{p}(\Omega; \Lambda^{k}), d\omega \in L^{p}(\Omega; \Lambda^{k+1}), \nu\lrcorner \omega = 0 
    \text{ on } \partial\Omega \right\rbrace,
   \end{align*}
   and similarly the spaces $W_{T}^{1,p}(\Omega; \Lambda^{k}) $ and $W_{N}^{1,p}(\Omega; \Lambda^{k}) .$ Also, we define, 
   $$W_{\delta, T}^{d,p}(\Omega; \Lambda^{k}) = \left\lbrace \omega \in W_{T}^{d,p}(\Omega; \Lambda^{k}) : \delta\omega = 0 \text{ in }
   \Omega \right\rbrace, $$
   and similarly $W_{\delta,T}^{1,p}(\Omega; \Lambda^{k}) .$ We also denote harmonic $k$-fields, harmonic $k$-fields with vanishing tangential component on the boundary and 
   harmonic $k$-fields with vanishing normal component on the boundary by the symbols $\mathcal{H}(\Omega,\boldsymbol{\Lambda^k} ), \mathcal{H}_{T}(\Omega,\boldsymbol{\Lambda^k} )$
    and $\mathcal{H}_{N}(\Omega,\boldsymbol{\Lambda^k} ),$ respectively. 
 \item We define the spaces $L^{\boldsymbol{p}}(\Omega,\boldsymbol{\Lambda^k} ),$ 
  $W^{1,{\boldsymbol{p}}}(\Omega,\boldsymbol{\Lambda^k} )$, $ W^{d,{\boldsymbol{p}}}(\Omega,\boldsymbol{\Lambda^k} ),$ and also the spaces 
  $W_{0}^{1,{\boldsymbol{p}}}(\Omega,\boldsymbol{\Lambda^k} ), W_{T}^{d,{\boldsymbol{p}}}(\Omega,\boldsymbol{\Lambda^k} ), 
  W_{\delta, T}^{d,{\boldsymbol{p}}}(\Omega,\boldsymbol{\Lambda^k} )$ etc, to be the corresponding product spaces. E.g.
   $$ W^{d,{\boldsymbol{p}}}(\Omega,\boldsymbol{\Lambda^k} ) = \prod_{i=1}^m W^{d, p_i}(\Omega, \Lambda^{k_i}).$$
   They are obviously also endowed with the corresponding product norms.
   When $p_i = \infty$ for all $1 \leq i \leq m ,$ we denote the corresponding spaces by $L^{\boldsymbol{\infty}}$ , $W^{1, \boldsymbol{\infty}}$ etc.

   \item In the same manner, $\displaystyle \boldsymbol{\omega}^{\nu} \boldsymbol{\rightharpoonup} \boldsymbol{\omega} \text{ in } W^{d,\boldsymbol{p}}\left(\Omega;\Lambda^{\boldsymbol{k -1}}\right)$
  will stand for a shorthand of $$ {\omega}^{\nu}_{i}\rightharpoonup \omega_{i}\text{ in }W^{d,p_i}\left(  \Omega ;\Lambda^{k_i -1}\right) \quad 
  ( \stackrel{\ast}{\rightharpoonup} \text{ if } p_{i}= \infty ),$$
for all $1 \leq i \leq m ,$ and 
$\displaystyle f\left( \boldsymbol{d \omega}^{\nu}\right) \rightharpoonup f\left( \boldsymbol{d\omega} \right) \text{ in }\mathcal{D}'(\Omega)$ 
will mean 
$$ f\left( d\omega_1^{\nu}, \ldots, d\omega_m^{\nu}\right) \rightharpoonup f\left( d\omega_1, \ldots, d\omega_m\right) \text{ in }\mathcal{D}'(\Omega).$$
\end{itemize}
\end{enumerate}

\section{Notions of Convexity}\label{notionsofconvexity}

\subsection{Definitions}
We start with the different notions of convexity and affinity. From here onwards, we are going to employ the boldface multiindex notations 
quite freely (Section \ref{notations} lists in detail all the notations that are employed).
\begin{definition}
Let $1\leq k_i\leq n$ for all $1 \leq i \leq m $ and $\displaystyle f:\prod_{i=1}^m \Lambda^{k_i}\left(  \mathbb{R}^{n}\right)
\rightarrow\mathbb{R}.$\smallskip

(i) We say that $f$ is  \emph{vectorially ext. one convex}, if the function%
\[
g:t\mapsto g\left(  t\right)  =f\left(  \xi_1 +t\,\alpha\wedge\beta_1 , \xi_2 +t\,\alpha\wedge\beta_2, \ldots , \xi_m +t\,\alpha\wedge\beta_m \right)
\]
is convex  for every collection of $\xi_i \in\Lambda^{k_i}, \   1 \leq i \leq m$, $\alpha\in\Lambda^{1}$ and $\beta_i
\in\Lambda^{k_i-1}$ for all $1 \leq i \leq m$. If the function $g$ is affine we say that $f$ is \emph{vectorially ext.
one affine.}\smallskip

(ii) A Borel measurable and locally bounded function $f$ is said to be
\emph{vectorially ext. quasiconvex}, if for every bounded open set $\Omega,$%
\[
\frac{1}{\lvert \Omega \rvert}\int_{\Omega}f\left(  \xi_1 +d\omega_1(x) , \xi_2 +d\omega_2(x), \ldots , \xi_m +d\omega_m(x) \right)
 \geq f\left(  \xi_1 , \xi_2 , \ldots , \xi_m \right)
\]
for every collection of $\xi_i \in\Lambda^{k_i}$ and $\omega_i \in W_{0}^{1,\infty}\left(  \Omega;\Lambda^{k_i-1}\right)$ with $     1 \leq i \leq m.$ If
equality holds, we say that $f$ is \emph{vectorially ext. quasiaffine.}\smallskip

(iii) We say that $f$ is \emph{vectorially ext. polyconvex}, if there exists a convex
function $F$ such that%
\[
f\left( \boldsymbol{\xi}\right)  =F\left(  T( \boldsymbol{\xi}) \right)  ,
\]
where $T(\boldsymbol{\xi})$ stands for the vector with components $\boldsymbol{\xi^{\alpha}} $, where $\boldsymbol{\alpha}$ 
varies over all possible choices such that $1 \leq  \lvert \boldsymbol{k\alpha} \rvert \leq n.$ (see section \ref{notations} for the notations).
If $F$ is affine, we say that $f$ is \emph{vectorially ext. polyaffine.}
\end{definition}
\begin{remark}
 (i)  The abbreviation ext. stands for
exterior, which refers to the exterior product in the first and third definitions and for the exterior derivative
for the second one.\smallskip

\noindent (ii) When $m=1,$ the notions of vectorial ext. polyconvexity, 
vectorial ext. quasiconvexity  and vectorial ext.  one convexity reduce to the ones introduced in \mbox{\cite{BandDacSil}}, namely, ext. polyconvexity, 
ext. quasiconvexity  and ext.  one convexity respectively.
\end{remark}

\begin{remark}
 The definition of \emph{vectorial ext. quasiconvexity} already appeared in 
 Iwaniec-Lutoborski \cite{iwaniec-lutoborski-null-lagrangian}, which the authors simply called \emph{quasiconvexity}. In the same article, the authors also introduce  
 another convexity notion, which they called \emph{polyconvexity}. But the definition of \emph{polyconvexity} introduced in 
 Iwaniec-Lutoborski \cite{iwaniec-lutoborski-null-lagrangian} is not equivalent to vectorial ext. polyconvexity. See remark \ref{polyconvexityremark} for more on this.
\end{remark}

\begin{remark}
 When $k_{i} = 1$ for all $1 \leq i \leq m,$ for each $\boldsymbol{\xi} \in \boldsymbol{\Lambda^{k}}$, by identifying $\xi_{i} \in \Lambda^{1}$ as 
 the $i$-th row, $\boldsymbol{\xi}$ can be 
 written as a $m\times n$ matrix. With this identification, the notions of vectorial ext. polyconvexity, 
vectorial ext. quasiconvexity  and vectorial ext.  one convexity are exactly the notions of polyconvexity, quasiconvexity and rank one convexity, respectively. 
\end{remark}

By requiring these properties to hold for each \emph{factor} while the others are kept fixed, we can define the corresponding `separate convexity' notions. 
\begin{definition}
 Let $1\leq k_i\leq n$ for all $1 \leq i \leq m $ and $\displaystyle f:\prod_{i=1}^m \Lambda^{k_i}\left(  \mathbb{R}^{n}\right)
\rightarrow\mathbb{R}.$\smallskip

(i) We say that $f$ is  \emph{separately ext. one convex} or \emph{ext. one convex with respect to each factor}, if for 
every $1 \leq i \leq m,$ the function $g_{i}: \Lambda^{k_{i}} \rightarrow \mathbb{R},$ given by, %
$$ g_{i}(\xi ) = f(\eta_{1}, \ldots, \eta_{i-1}, \xi, \eta_{i+1}, \ldots, \eta_{m}) $$
is ext. one convex for every collection of $\eta_{j} \in \Lambda^{k_{j}} ,$ $1 \leq j \leq m,$ $ j \neq i.$ We say $f$ is  \emph{separately ext. one affine} if 
$g_{i}$s are  ext. one affine.\smallskip

(ii) A Borel measurable and locally bounded function $f$ is said to be
\emph{separately ext. quasiconvex} or \emph{ext. quasiconvex with respect to each factor}, if for every $1 \leq i \leq m,$ the function $g_{i}: \Lambda^{k_{i}} \rightarrow \mathbb{R},$ given by, %
$$ g_{i}(\xi ) = f(\eta_{1}, \ldots, \eta_{i-1}, \xi, \eta_{i+1}, \ldots, \eta_{m}) $$
is ext. quasiconvex for every collection of $\eta_{j} \in \Lambda^{k_{j}} ,$ $1 \leq j \leq m,$ $ j \neq i.$ We say $f$ is  \emph{separately ext. quasiaffine} if 
$g_{i}$s are  ext. quasiaffine.\smallskip 

(iii) We say that $f$ is \emph{separately ext. polyconvex} or \emph{ext. polyconvex with respect to each factor}, if for every $1 \leq i \leq m,$ the function $g_{i}: \Lambda^{k_{i}} \rightarrow \mathbb{R},$ given by, %
$$ g_{i}(\xi ) = f(\eta_{1}, \ldots, \eta_{i-1}, \xi, \eta_{i+1}, \ldots, \eta_{m}) $$
is ext. polyconvex for every collection of $\eta_{j} \in \Lambda^{k_{j}} ,$ $1 \leq j \leq m,$ $ j \neq i.$ We say $f$ is  \emph{separately ext. polyaffine} if 
$g_{i}$s are  ext. polyaffine.

\end{definition}

Note that the notions of \emph{separately ext. one affine}, \emph{separately ext. quasiaffine} and \emph{separately ext. polyaffine} are all equivalent. It is easy to see from the definitions, using the relations between ext. polyconvexity, ext. quasiconvexity and 
ext. one convexity (cf. Theorem 2.8(i) in \cite{BandDacSil}), that 
\begin{itemize}
 \item $f \text{ \emph{vectorially ext. one convex} } \Rightarrow f \text{ \emph{separately ext. one convex}.} $
 \item $\!\begin{aligned}[t]
           f \text{ \emph{vectorially ext. quasiconvex} } &\Rightarrow f \text{ \emph{separately ext. quasiconvex} } \\
                                                           &\Rightarrow f \text{ \emph{separately ext. one convex}.}
          \end{aligned}$
\item $\!\begin{aligned}[t]
          f \text{ \emph{vectorially ext. polyconvex}} &\Rightarrow f \text{ \emph{separately ext. polyconvex}} \\
                                                        &\Rightarrow f \text{ \emph{separately ext. quasiconvex}} \\
                                                         &\Rightarrow f \text{ \emph{separately ext. one convex}}.
         \end{aligned}$  
\end{itemize}
Note that the notion of a separately convex function is very different. For 
$f$ to be separately convex, we require convexity with respect to each \emph{component}, not each \emph{factor}. All the convexity notions above implies separate convexity of $f$, but 
none is implied by it. As an example, the function defined by the multiplication of all the components of all the factors, i.e $f(\xi_{1}, \ldots,  \xi_{m}) = \prod\limits_{i=1}^{m} \prod\limits_{I \in \mathcal{T}^{k_{i}}} \xi_{i}^{I},$ is clearly separately convex, but 
not separately ext. one convex and thus none of the others as well.\smallskip

As in \cite{BandDacSil}, we can use Hodge duality to extend these notions of convexity to the ones related to interior product and $\delta$-operator.
We shall discuss vectorial ext. convexity properties only. Vectorial int. convexity notions can be handled analogously.

\subsection{Basic Properties}
The different notions of vectorial ext. convexity are related as follows.
\begin{theorem}
\label{implications for vectorial ext convexity} Let $\displaystyle f:\boldsymbol{\Lambda^k}\rightarrow\mathbb{R}.$
 Then%
\begin{align*}
 f \text{convex} \Rightarrow f\text{  vectorially ext. polyconvex }&\Rightarrow\text{
}f\text{ vectorially ext. quasiconvex }\\
&\Rightarrow\text{ }f\text{ vectorially ext. one convex.}
\end{align*}
Moreover if $f:\boldsymbol{\Lambda^{k}}\left(  \mathbb{R}^{n}\right)  \rightarrow
\mathbb{R}$ is vectorially ext. one convex, then $f$ is locally Lipschitz.
\end{theorem}

\begin{proof}
The proof is very similar to the proof of theorem 2.8 in \cite{BandDacSil} (see \cite{silthesis} for a more detailed proof). We only mention here the essential differences. 
The implication that 
$ f$ convex implies $f$ vectorially ext. polyconvex is trivial.\smallskip

\noindent To prove the implication,
$$f\text{  vectorially ext. polyconvex }\Rightarrow\text{
}f\text{ vectorially ext. quasiconvex }, $$
the argument using Jensen's inequality is exactly the same as in theorem 2.8 in \cite{BandDacSil}, as soon as we show
$$ \int_{\Omega}\left(  \boldsymbol{\xi}+\boldsymbol{d\omega}\right)  ^{\boldsymbol{\alpha}}=\boldsymbol{\xi}^{\boldsymbol{\alpha}}\operatorname*{meas}\left( \Omega \right),$$
for any $\boldsymbol{\xi} \in \boldsymbol{\Lambda^{k}},$ for any $\boldsymbol{\omega} \in 
W_{0}^{1,{\boldsymbol{\infty}}}(\Omega,\boldsymbol{\Lambda^k} )$ and for any multiindex $\boldsymbol{\alpha}.$ We prove this using induction over 
 $ \displaystyle \lvert \boldsymbol{\alpha} \rvert.$ The case $ \displaystyle \lvert \boldsymbol{\alpha} \rvert= 1$ easily follows from integration by parts. So we assume 
 $ \displaystyle \lvert \boldsymbol{\alpha} \rvert > 1.$ Thus, there exists $i$ such that $\alpha_{i} \geq 2.$
 Now, we have, 
 \begin{align*}
  \left(  \boldsymbol{\xi}+\boldsymbol{d\omega}\right)^{\boldsymbol{\alpha}}  &  =\xi_{i}\wedge  \left(  \boldsymbol{\xi}+\boldsymbol{d\omega}\right)^{\boldsymbol{\beta}}
+d\omega_{i}\wedge  \left(  \boldsymbol{\xi}+\boldsymbol{d\omega}\right)^{\boldsymbol{\beta}} \smallskip\\
&  = \xi_{i}\wedge  \left(  \boldsymbol{\xi}+\boldsymbol{d\omega}\right)^{\boldsymbol{\beta}}  +d\left[  \omega_{i}\wedge 
\left(  \boldsymbol{\xi}+\boldsymbol{d\omega}\right)^{\boldsymbol{\beta}}  \right],
\end{align*}
where $\boldsymbol{\beta}$ is a multiindex with $\beta_{i} = \alpha_{i} -1$ and $\beta_{j} = \alpha_{j}$ for all $1 \leq j \leq m, i\neq j.$ Since 
$\displaystyle \lvert \boldsymbol{\beta} \rvert =  \lvert \boldsymbol{\alpha} \rvert -1,$ integrating the above and using induction for the first integral and integration by parts along with 
the fact that 
$\omega_{i} = 0$ on $\partial\Omega$ for the second, we obtain the result.\smallskip

\noindent The implication
$$ f\text{ vectorially ext. quasiconvex }\Rightarrow\text{ }f\text{ vectorially ext. one convex},$$ is proved by the same arguments as in theorem 2.8 in \cite{BandDacSil}, using 
lemma 2.7 in \cite{BandDacSil} for each factor.\smallskip

\noindent The fact that $f$ is locally Lipschitz follows once again from the observation that
any separately ext. one convex function is separately convex. 
\end{proof}\smallskip 

\noindent We can have another formulation of vectorial ext. polyconvexity. The proof of which is similar to Proposition 2.14 in \cite{BandDacSil} and is omitted.
\begin{proposition}
Let $\displaystyle f:\boldsymbol{\Lambda^k}\rightarrow\mathbb{R}.$
Then,
the function $f$ is ext. polyconvex if and only if, for every $ \displaystyle \boldsymbol{\xi}\in 
\boldsymbol{\Lambda^k} $, there 
exist $ c_{ \alpha} = c_{ \alpha}(\boldsymbol{\xi}) \in\Lambda^{ \lvert \boldsymbol{k\alpha} \rvert }(\mathbb{R}^n),$ for every $\boldsymbol{\alpha}$ 
with  $0 \leq \lvert \boldsymbol{k\alpha} \rvert \leq n$, such that  %
\[
f\left(  \boldsymbol{\eta}\right)  \geq f\left(  \boldsymbol{\xi}\right)  +\sum_{\boldsymbol{\alpha}  }\left\langle c_{s}\left( \boldsymbol{\xi}\right)  ;
\boldsymbol{\eta}^{\boldsymbol{\alpha}}-\boldsymbol{\xi}^{\boldsymbol{\alpha}}\right\rangle ,\quad\text{for every }\boldsymbol{\eta} \in\boldsymbol{\Lambda^{k}}.
\]
\end{proposition}
\begin{remark}\label{polyconvexityremark}
 Comparison with the definition of \emph{polyconvexity} introduced in definition $10.1$ in 
 Iwaniec-Lutoborski \cite{iwaniec-lutoborski-null-lagrangian}, one easily sees that their definition allows only the case $\alpha_{i} \in \lbrace 0, 1 \rbrace$  for all $1 \leq i \leq m $. 
We remark that unless $k_i$s are all odd integers, these two classes of 
polyconvex functions do not coincide and ours is strictly larger. For example, identifying $\mathbb{R}$ with $\Lambda^{n},$ the function $f_{1}: \Lambda^{k_{1}}\times \Lambda^{k_{2}} \rightarrow \mathbb{R} $ given by,
 $$ f_{1}(\xi_{1}, \xi_{2}) = \langle c; \xi_{1} \wedge \xi_{2} \rangle \quad \text{ for every }\xi_{1} \in \Lambda^{k_{1}}, \xi_{2} \in \Lambda^{k_{2}}$$
 where $c \in \Lambda^{(k_{1} + k_{2})}$ is a constant form, is \emph{polyaffine} in the sense of Iwaniec-Lutoborski \cite{iwaniec-lutoborski-null-lagrangian} and also 
 \emph{vectorially ext. polyaffine}. However, the function $f_{2}: \Lambda^{k_{1}}\times \Lambda^{k_{2}} \rightarrow \mathbb{R} $ given by,
 $$ f_{2}(\xi_{1}, \xi_{2}) = \langle c; \xi_{1} \wedge \xi_{1} \rangle \quad \text{ for every }\xi_{1} \in \Lambda^{k_{1}}, \xi_{2} \in \Lambda^{k_{2}}$$
 where $c \in \Lambda^{2k_{1}}$ is a constant, is \emph{vectorially ext. polyaffine}, but not \emph{polyaffine} in the sense of 
 Iwaniec-Lutoborski \cite{iwaniec-lutoborski-null-lagrangian}, unless $k_{1}$ is odd or $2k_{1} > n.$ Note that the crucial point is the self-wedge product, 
 not the fact that $f_{2}$ is independent of $\xi_{2}.$
 $f_{1}+ f_{2}$ is also \emph{vectorially ext. polyaffine}, but not \emph{polyaffine} in the sense of 
 Iwaniec-Lutoborski \mbox{\cite{iwaniec-lutoborski-null-lagrangian}}. Note also that it is easy to see, by integrating by parts that $f_{1}, f_{2}$ and $f_{1} + f_{2}$ are 
 all \emph{vectorially ext. quasiaffine} and hence are also \emph{quasiaffine} in the 
 sense of Iwaniec-Lutoborski \cite{iwaniec-lutoborski-null-lagrangian}. Also, when $m=1$, i.e there is only one differential form, reducing the problem to the functionals having 
 the form $\int_{\Omega} f(d\omega),$ their definition of 
  \emph{polyconvexity} coincide with usual convexity. On the other hand, when $m=1$, \emph{vectorial ext. polyconvexity} reduces to \emph{ext. polyconvexity}, which 
  is much weaker than convexity and has been discussed 
  in detail in \cite{BandDacSil}. 
\end{remark}

\subsection{The quasiaffine case}
We now prove the basic characterization theorem for vectorially ext. quasiaffine functions. In the special case when $k_{i}=1$ for all $1 \leq i \leq m,$ this immediately implies 
classical theorem of Ball \cite{Ballquasiaffine} with a new proof. In a sense, this theorem also `explains' the appearance of determinants and adjugates in the classical theorem. Determinants and adjugates appear as they are 
precisely the `wedge products' in the classical case.  
\begin{theorem}
\label{Thm vectorial ext. quasiaffine}
Let $\displaystyle f:\boldsymbol{\Lambda^k}\rightarrow\mathbb{R}.$
The following statements are then equivalent.\smallskip

\textbf{(i)} $f$ is vectorially ext. polyaffine.\smallskip

\textbf{(ii)} $f$ is vectorially ext. quasiaffine.\smallskip

\textbf{(iii)} $f$ is vectorially ext. one affine.\smallskip

\textbf{(iv)} There exist $ c_{ \alpha}\in\Lambda^{ \lvert \boldsymbol{k\alpha} \rvert }(\mathbb{R}^n),$ for every $\boldsymbol{\alpha}
= (\alpha_1, \ldots, \alpha_m )$ such that $ 0 \leq \alpha_i \leq \left[ \frac{n}{k_i} \right]$ 
for all $1 \leq i \leq m$ and $0 \leq \lvert \boldsymbol{k\alpha} \rvert \leq n$, such that for every $ \displaystyle \boldsymbol{\xi}\in 
\boldsymbol{\Lambda^k} $, %
\[
f\left(  \boldsymbol{\xi} \right)  =\sum_{ \substack{ \alpha,  \\ 0 \leq \lvert \boldsymbol{k\alpha} \rvert \leq n } }\left\langle c_{\alpha} %
;\boldsymbol{\xi^{\alpha}}\right\rangle .
\]

\end{theorem}
\begin{remark}
 If $k_i = 1$ for all $1 \leq i \leq m $, then this theorem recovers the characterization theorem for quasiaffine functions in classical vectorial calculus of variation as a special case. Indeed, let $X \in \mathbb{R}^{m \times n}$ 
 be a matrix, then setting  $ \displaystyle  \xi_i = \sum_{j = 1}^{n} X_{ij} e^j $ for all $1 \leq i \leq m , $ we recover exactly the classical results (cf. 
 Theorem $5.20$ in \cite{DCV2}).
\end{remark}

\begin{proof}
 $ (i) \Rightarrow (ii) \Rightarrow (iii)$ follows from Theorem \ref{implications for vectorial ext convexity}. $(iv) \Rightarrow (i)$ is immediate from the definition of vectorial ext. polyconvexity. So we only 
 need to show $(iii) \Rightarrow (iv).$\smallskip
 
 We show this by induction on $m$. Clearly, for $m = 1$, this is just the characterization theorem for ext. one affine functions, given in 
 theorem  3.3 in \cite{BandDacSil}. We assume the result to be true for 
 $m \leq p-1$ and show it for $m=p.$  
 Now since $f$ is vectorially ext. one affine, it is separately ext. one affine and using ext. one affinity with respect to $\xi_p$, keeping the other variables 
 fixed, we obtain,
 $$ f\left(  \boldsymbol{\xi} \right)  = \sum_{s=1}^{[\frac{n}{k_p}]} \langle c_{s}(\xi_1,\ldots, \xi_{p-1}) ; \xi_p^s \rangle , $$
 where for each $1 \leq s \leq [\frac{n}{k_p}] $, the functions $ \displaystyle  c_{s}: \prod_{i=1}^{p-1} \Lambda^{k_i} \rightarrow \Lambda^{sk_p}$ are such that 
 the map  $(\xi_1,\ldots, \xi_{p-1}) \mapsto  f\left( \xi_1,\ldots, \xi_{p-1}, \xi_p  \right)$ is vectorially ext. one affine for any $\xi_p \in \Lambda^{k_p}.$
 Arguing by degree of homogeneity, this implies that  for each $1 \leq s \leq [\frac{n}{k_p}] $, every component $c_S^{I}$ is 
 vectorially ext. one affine, i.e  $(\xi_1,\ldots, \xi_{p-1}) \mapsto c_{s}^{I}(\xi_1,\ldots, \xi_{p-1})$ is vectorially ext. one affine for any $I \in \mathcal{T}_{sk_p}.$
 Applying the induction hypothesis to each of these components and multiplying out, we indeed obtain the desired result.
\end{proof}
\begin{remark}
 Note that since the proof of Theorem 3.3 in \cite{BandDacSil} does not use the classical result about quasiaffine functions, this really yields a new proof even in the special case 
  of $k_{i}=1$ for all $1 \leq i \leq m.$
\end{remark}

\section{Weak lower semicontinuity}\label{weaklowersemicontinuity}
Now we investigate the relationship between vectorial ext. quasiconvexity of the integrand and sequential weak lower semicontinuity of the integral functionals. 
\subsection{Necessary condition} 
\begin{theorem}[Necessary condition]\label{necessary condition semicontinuity}
Let $\Omega\subset\mathbb{R}^n$ be open, bounded. 
Let $f:\Omega\times\boldsymbol{\Lambda^{k-1}}\times\boldsymbol{\Lambda^k}\rightarrow \mathbb{R}$ be a Carath\'eodory
 function satisfying, for almost all $x\in\Omega$ and for all $ ( \boldsymbol{\omega}, \boldsymbol{\xi} ) \in \boldsymbol{\Lambda^{k-1}}\times\boldsymbol{\Lambda^k},$
\begin{equation}\label{caratheodory}
|f(x,\boldsymbol{\omega}, \boldsymbol{\xi})|\leqslant a(x) +b(\boldsymbol{\omega}, \boldsymbol{\xi}),
\end{equation}
 where 
 $a\in L^1\left(\mathbb{R}^n\right)$, $b\in C\left(\boldsymbol{\Lambda^{k-1}}\times\boldsymbol{\Lambda^k}\right)$ is non-negative. 
 Let the functional $I:W^{d,\boldsymbol{\infty}}\left(\Omega;\boldsymbol{\Lambda^{k-1}}
 \right)\rightarrow \mathbb{R}$, defined by
$$
I(\boldsymbol{\omega}):=\int_{\Omega}f\left(x,\boldsymbol{\omega}(x), \boldsymbol{d\omega}(x)\right)\,dx,\text{ for all }\boldsymbol{\omega}\in W^{d,\boldsymbol{\infty}}\left(\Omega;\boldsymbol{
\Lambda^{k-1}}\right),
$$
be weak $\ast$ lower semicontinuous in $W^{d,\boldsymbol{\infty}}\left(\Omega;\boldsymbol{\Lambda^{k-1}}\right)$. Then, for almost all $x_0\in \Omega$ and for all $\boldsymbol{\omega_0}\in\boldsymbol{\Lambda^{k-1}}$, $\boldsymbol{\xi_0}\in\boldsymbol{\Lambda^{k}}$ and 
$\boldsymbol{\phi}\in W^{d,\boldsymbol{\infty}}\left(D;\boldsymbol{\Lambda^{k}}\right)$,
$$
\int_{D}f\left(x_0,\boldsymbol{\omega_0},\boldsymbol{\xi_0}+\boldsymbol{d\phi}(x)\right)\,dx\geqslant f\left(x_0,\boldsymbol{\omega_0},\boldsymbol{\xi_0}\right),
$$
where $D=(0,1)^n\subset\mathbb{R}^n$. In particular, $\boldsymbol{\xi} \mapsto 
f\left(x,\boldsymbol{\omega}, \boldsymbol{\xi}\right)$ is vectorially ext. quasiconvex for a.e $x \in \Omega$ and for every $\boldsymbol{\omega} \in \boldsymbol{\Lambda^{k-1}}.$  
\end{theorem}
\begin{remark}\label{remarkinftynecessaryforp}
 Since $I$ being weak $\ast$ lower semicontinuous in $W^{d,\boldsymbol{\infty}}\left(\Omega;\boldsymbol{\Lambda^{k-1}}\right)$ is a necessary condition for $I$ to be
 weak lower semicontinuous in $W^{d,\boldsymbol{p}}\left(\Omega;\boldsymbol{\Lambda^{k-1}}\right)$ for any 
 $\boldsymbol{p},$ $f$ being vectorially ext. quasiconvex is a necessary condition for 
 weak lower semicontinuity in $W^{d,\boldsymbol{p}}\left(\Omega;\boldsymbol{\Lambda^{k-1}}\right)$ as well. 
\end{remark}
The proof of this result is  a long but straightforward adaptation of the classical  proof (due to Acerbi-Fusco 
\mbox{\cite{AcerbiFuscosemicontinuity}}) for the gradient case (cf. Theorem 3.15 in \cite{DCV2})  and is omitted. See \cite{silthesis} for a detailed proof.

\subsection{Lower semicontinuity for quasiconvex functions without lower order terms}

We now turn to sufficient conditions for sequential weak lower semicontinuity. We begin by defining the appropriate growth conditions.
\begin{definition}[Growth condition I]
 Let $\Omega\subset\mathbb{R}^n$ be open, bounded and let $f:\boldsymbol{\Lambda^k}\rightarrow \mathbb{R}.$  
 Let $\boldsymbol{p}$ be given.\smallskip
  
  \noindent $f$ is said to be of growth $\left(\mathcal{C}_{\boldsymbol{p}}\right)$, if 
  for every $\boldsymbol{\xi} =  (\xi_1, \ldots, \xi_m ) \in \boldsymbol{\Lambda^{k}}$,  $f$ satisfies, 
  \begin{align*}
    - \alpha \left( 1 + \sum_{i=1}^{m}  G^{l}_{i}( \xi_{i} ) \right)  \leq f(\boldsymbol{\xi}) \leq \alpha \left( 1  +  \sum_{i=1}^{m} G^{u}_{i}(\xi_i) \right) ,\tag{ $\mathcal{C}_{\boldsymbol{p}} $}
  \end{align*}
where $\alpha > 0$ is a constant and the functions $G^{l}_{i}$s in the lower bound and the functions $G^{u}_{i}$s in the upper bound has the following form:
\begin{itemize}
 \item If $p_i =1$, then,
                   \begin{align*}
                    G^{l}_{i}(\xi_i) = G^{u}_{i}(\xi_i) = \alpha_{i} \lvert \xi_{i} \rvert \qquad \qquad \text{ for some constant } \alpha_i \geq 0. 
                    \end{align*}
\item If $1 < p_i < \infty$, then,
\begin{align*}
 G^{l}_{i}(\xi_i) = \alpha_{i}\lvert \xi_{i} \rvert^{q_i} \quad \text{ and } \quad 
 G^{u}_{i}(\xi_i) = \alpha_{i}\lvert \xi_{i} \rvert^{p_i} , 
\end{align*}
for some $ 1 \leq q_i < p_i $ and for some constant $ \alpha_i \geq 0.$ 
\item If $p_i = \infty$, then, 
\begin{align*}
 G^{l}_{i}(\xi_i) = G^{u}_{i}(\xi_i) = \eta_{i}\left( \lvert \xi_{i} \rvert \right)  . 
\end{align*}
for some nonnegative, continuous, increasing function  $\eta_i$.
\end{itemize}
\end{definition}
Now we need a lemma which is essentially an analogue of the result relating quasiconvexity with $W^{1,p}$-quasiconvexity in the classical case (see Ball-Murat \cite{ballmurat}) 
and is proved in a similar manner.   
\begin{lemma}[$W^{d,\boldsymbol{p}}$-quasiconvexity]\label{equivalence of quasiconvexity}
  Let $\Omega\subset\mathbb{R}^n$ be open, bounded, smooth. 
 Let $f:\boldsymbol{\Lambda^k}\rightarrow \mathbb{R}$ satisfy, for every $\boldsymbol{\xi} =  (\xi_1, \ldots, \xi_m ) \in \boldsymbol{\Lambda^{k}}$,
 \begin{align*}
    f(\boldsymbol{\xi}) \leq \alpha \left( 1  +  \sum_{i=1}^{m} G^{u}_{i}(\xi_i) \right) ,
  \end{align*}
  where $\alpha > 0$ is a constant and the functions $G^{u}_{i}$s are as defined above, with a given $\boldsymbol{p}.$ 
Then the following are equivalent.
\begin{enumerate}
 \item[(i)] $f$ is vectorially ext. quasiconvex.
 \item[(ii)] For every $\boldsymbol{q}$ such that $p_{i} \leq q_{i} \leq \infty$ for every $i=1, \ldots, m$, we have,
 
 \begin{align*}
              \frac{1}{\operatorname*{meas}(\Omega)} \int_{\Omega} f(\boldsymbol{\xi} + d \boldsymbol{\phi}) \geq f(\boldsymbol{\xi} ),
             \end{align*}
for every $\boldsymbol{\phi} \in W^{d,\boldsymbol{q}}_{T}\left(\Omega;\boldsymbol{\Lambda^{k-1}}\right).$
\end{enumerate}
\end{lemma}
\begin{proof}
 For any $\boldsymbol{\phi} \in W^{d,\boldsymbol{q}}_{T}\left(\Omega;\boldsymbol{\Lambda^{k-1}}\right) $, we find 
 $\lbrace\boldsymbol{\phi^{\nu}} \rbrace \subset C_{c}^{\boldsymbol{\infty}}\left(\Omega;\boldsymbol{\Lambda^{k-1}}\right)$ such that 
 $\lbrace \boldsymbol{\phi^{\nu}} \rbrace$ is uniformly bounded in $W^{d,\boldsymbol{p}}\left(\Omega;\boldsymbol{\Lambda^{k-1}}\right)$ and $ \boldsymbol{d\phi^{\nu}}
 \rightarrow \boldsymbol{d\phi}$ for a.e $x \in \Omega.$ Since $f$ is continuous, applying Fatou's lemma we obtain,
 \begin{multline*}
  \liminf_{\nu \rightarrow \infty}  \int\limits_{\Omega} \left[ \alpha \left( 1  +  \sum_{i=1}^{m} G^{u}_{i}(\phi^{\nu}_i) \right) - f(\boldsymbol{\xi} + d \boldsymbol{\phi^{\nu}}) \right] 
  \\\geq \int_{\Omega} \left[ \alpha \left( 1  +  \sum_{i=1}^{m} G^{u}_{i}(\phi_i) \right) - f(\boldsymbol{\xi} + d \boldsymbol{\phi}) \right]. 
 \end{multline*}
  Since $\displaystyle \lim_{\nu \rightarrow \infty}   \int\limits_{\Omega} \left( 1  +  \sum_{i=1}^{m} G^{u}_{i}(\phi^{\nu}_i) \right) =  
\int_{\Omega}  \left( 1  +  \sum_{i=1}^{m} G^{u}_{i}(\phi_i) \right),$ by dominated convergence theorem, vectorial ext. quasiconvexity of $f$ yields the result.
\end{proof}\smallskip

We now generalize an elementary proposition from convex analysis in this setting. The proof is straightforward and is just a matter of  
iterating the argument in the proof of Proposition 2.32 in \cite{DCV2}. So we provide only a brief sketch.  
\begin{proposition}\label{semicontinuity p lipscitz inequality without infty}
 Let $\boldsymbol{p} = (p_1,\ldots, p_m)$ with $1\leq p_i < \infty$ for all $1 \leq i \leq m $ and  let $\Omega\subset\mathbb{R}^n$ be open, bounded, smooth. 
 Let $f:\boldsymbol{\Lambda^k}\rightarrow \mathbb{R}$ be separately convex and satisfy, for every $\boldsymbol{\xi} =  (\xi_1, \ldots, \xi_m ) \in \boldsymbol{\Lambda^{k}}$,
 \begin{align*}
    \lvert f(\boldsymbol{\xi}) \rvert \leq \alpha \left( 1  +  \sum_{i=1}^{m} \lvert \xi_{i} \rvert^{p_i} \right) ,
  \end{align*}
  where $\alpha > 0$ is a constant. Then there exist constants $\beta_{i} > 0, i=1,\ldots, m$ such that
  \begin{align*}
   \lvert f(\boldsymbol{\xi}) -  f(\boldsymbol{\zeta}) \rvert \leq \sum_{i=1}^{m} \beta_{i} \left(  1 + \sum_{j=1}^{m} \left( \lvert \xi_{j} \rvert^{\frac{p_{j}}{p'_{i}} }  + 
   \lvert \zeta_{j} \rvert^{\frac{p_{j}}{p'_{i}}} \right)  \right)  \lvert  \xi_{i} - \zeta_{i}\rvert ,
  \end{align*}

  for every $\boldsymbol{\xi} =  (\xi_1, \ldots, \xi_m ), \boldsymbol{\zeta} =  (\zeta_1, \ldots, \zeta_m ) \in \boldsymbol{\Lambda^{k}},$
where $p'_{i}$ is the H\"{o}lder conjugate of exponent of $p_{i}.$
\end{proposition}

\begin{proof}
 We know that for any convex function $g:\mathbb{R} \rightarrow \mathbb{R}$, we have, for every $\lambda > \mu > 0$ and for every $t \in \mathbb{R}$,
 $$ \frac{g(t \pm \mu) - g(t)}{\mu} \leq \frac{g(t \pm \lambda) - g(t)}{\lambda}. $$
Now let
 $$ g_{i}^{I}(t) := f(t, \widetilde{\boldsymbol{\xi}}^{i,I}),$$ where $\widetilde{\boldsymbol{\xi}}^{i,I}$ 
 is the vector whose components are precisely all the components of $\boldsymbol{\xi}$ except $\xi_{i}^{I}.$
 Choosing $\mu =  \zeta_{i}^{I} - \xi_{i}^{I}$ and $ \displaystyle \lambda = 1 + \lvert \xi_{i} \rvert + \lvert \zeta_{i} \rvert + 
 \sum_{j\neq i}  \lvert \xi_{j} \rvert^{\frac{p_{j}}{p_{i}}} , $ we obtain,
 $$  g(\zeta_{i}^{I}) - g(\xi_{i}^{I}) = g( \xi_{i}^{I} + \mu) - g(\xi_{i}^{I}) \leq \mu \frac{g(\xi_{i}^{I} + \lambda) - g(\xi_{i}^{I})}{\lambda}.$$
 The same can be done for $  g(\xi_{i}^{I}) -  g(\zeta_{i}^{I})$ as well. Now, using the growth conditions and writing 
 $f(\boldsymbol{\xi})-f(\boldsymbol{\zeta})$ as a sum of differences, the estimate follows.  
\end{proof}
\begin{remark}
 A similar looking inequality was claimed in Iwaniec-Lutoborski (\cite{iwaniec-lutoborski-null-lagrangian}, (10.3)), which however is easily seen to be false.  Take for example,
 the function $W:\Lambda^{k}\times \Lambda^{n-k} \rightarrow \Lambda^{n},$ defined by $W(\xi, \eta) = \xi \wedge \eta.$  It is easy to see that 
 $\left\lvert W(\xi, \eta)\right\rvert \leq \widetilde{C} \left( \left\lvert \xi\right\rvert^{2} + \left\lvert \eta\right\rvert^{2}\right),$ for some constant 
 $\widetilde{C}>0.$ Now, choose $\xi_{1}, \xi_{2} \in \Lambda^{k}$ 
 and $\eta \in  \Lambda^{n-k}$ such that $\left(\xi_{1} - \xi_{2} \right)\wedge \eta \neq 0.$ Now, for any $\lambda \in \mathbb{R},$ applying the inequality for the points 
 $\left( \xi_{1}, \lambda \eta \right)$ and $\left( \xi_{2}, \lambda \eta \right)$ gives 
 $$ \left\lvert \lambda \right\rvert \left\lvert \left(\xi_{1} - \xi_{2} \right)\wedge \eta\right\rvert \leq C 
 \left( \left\lvert \xi_{1}\right\rvert + \left\lvert \xi_{2}\right\rvert\right)^{(2-1)} \left\lvert \xi_{1}- \xi_{2}\right\rvert .$$
 Letting $\lvert \lambda \rvert \rightarrow \infty,$ it is clear that no such constant $C > 0$ can exist.\end{remark}
This proposition can be easily generalized to cover the case where some of the $p_{i}$s can be $\infty$ as well. 
\begin{proposition}\label{semicontinuity p lipscitz inequality}
 Let $ 0 \leq r \leq m$ be an integer. Let $\boldsymbol{p} = (p_1,\ldots, p_m)$
 where $1\leq p_i < \infty$ for all $1 \leq i \leq r $ and $p_{r+1} = \ldots = p_{m} = \infty.$ Let $\Omega\subset\mathbb{R}^n$ be open, bounded, smooth. 
 Let $f:\boldsymbol{\Lambda^k}\rightarrow \mathbb{R}$ be separately convex and satisfy, for every $\boldsymbol{\xi} =  (\xi_1, \ldots, \xi_m ) \in \boldsymbol{\Lambda^{k}}$,
 \begin{align*}
    \lvert f(\boldsymbol{\xi}) \rvert \leq \alpha \left( 1  +  \sum_{i=1}^{r} \lvert \xi_{i} \rvert^{p_i} + \sum_{i=r+1}^{m} \eta_{i}\left( \lvert \xi_{i} \rvert \right) \right) ,
    \end{align*}
    where $\alpha > 0$ is a constant and $\eta_i$s are some nonnegative, continuous, increasing functions.  Let $$ Q :=[-C, C]^{\sum\limits_{i = r+1}^{m} \binom{n}{k_{i}} } 
    \subset \subset \prod_{i= r+1}^{m}\Lambda^{k_{i}}$$ be a cube and define
    $$ K := \Lambda^{k_{1}}\times\ldots\times\Lambda^{k_{r}}\times Q . $$
Then there exist constants $\beta_{i}= \beta_{i}(K) > 0, i=1,\ldots, m$ such that
  \begin{align}\label{p lipscitz inequality}
   \lvert f(\boldsymbol{\xi}) -  f(\boldsymbol{\zeta}) \rvert \leq \sum_{i=1}^{r} \beta_{i} & \left(  1 + \sum_{j=1}^{r} \left( \lvert \xi_{j} \rvert^{\frac{p_{j}}{p'_{i}} }  + 
   \lvert \zeta_{j} \rvert^{\frac{p_{j}}{p'_{i}}} \right)  \right)  \lvert  \xi_{i} - \zeta_{i}\rvert \notag 
    \\ &+ \sum_{i=r+1}^{m} \beta_{i} \left(  1 + \sum_{j=1}^{r} \left( \lvert \xi_{j} \rvert^{p_{j}}    + 
   \lvert \zeta_{j} \rvert^{p_{j}} \right)  \right)  \lvert  \xi_{i} - \zeta_{i}\rvert ,
  \end{align}

  for every $\boldsymbol{\xi} =  (\xi_1, \ldots, \xi_m ), \boldsymbol{\zeta} =  (\zeta_1, \ldots, \zeta_m ) \in K,$
where $p'_{i}$ is the H\"{o}lder conjugate of exponent of $p_{i}.$
\end{proposition}
\begin{remark}
Clearly, when $r=m$, the last term  and when $r=0$, the first term  is not present in the inequality \eqref{p lipscitz inequality}.
Also the assumption on the naming of the variable is clearly not a restriction at all, since we can always relabel the variables.
\end{remark}
\begin{proof}
 We split $ f(\boldsymbol{\xi}) -  f(\boldsymbol{\zeta})$ as a sum of $ f(\boldsymbol{\xi}) -  f(\zeta_{1},\ldots,\zeta_{r}, \xi_{r+1},\ldots, \xi_{m})$ and 
 $  f(\zeta_{1},\ldots,\zeta_{r}, \xi_{r+1},\ldots, \xi_{m}) -  f(\boldsymbol{\zeta}).$ Now the first term is 
 estimated using proposition \mbox{\ref{semicontinuity p lipscitz inequality without infty}}, using the fact that $\eta_{i}$s are
  bounded on $\left[-C,C \right]$ for $r+1\leq i \leq m.$ For the second term, we note that for any convex function $g:\mathbb{R} \rightarrow \mathbb{R}$,  for any 
  $x, y \in [-C,C]$, we have the estimate 
  $\lvert g(x) - g(y)\rvert \leq 2 \left(\max\limits_{ \lvert t \rvert \leq C+1} \left\lvert g(t) \right\rvert \right)\lvert x -y \rvert.$ Using separate convexity along with 
  this estimate, we obtain the result.  
\end{proof}\smallskip 

\noindent Now we need a decomposition lemma, which lets us replace a uniformly bounded sequence of exterior derivatives in $L^{\boldsymbol{p}}$ by a sequence with equiintegrable one, upto 
sets of small measure.
\begin{lemma}\label{decomposition lemma}
 Let $\boldsymbol{p} = (p_1,\ldots, p_m)$ where $1 < p_i < \infty$ for all $1 \leq i \leq m. $ Let $\Omega\subset\mathbb{R}^n$ be open, bounded, smooth and 
 $$ \boldsymbol{\omega^{r}} \rightharpoonup \boldsymbol{\omega} \quad \text{ in } W^{d,\boldsymbol{p}}\left(\Omega;\boldsymbol{\Lambda^{k-1}}\right), $$ Then there exist a 
 subsequence $\lbrace \boldsymbol{\omega^{s}}\rbrace$ and a sequence $\lbrace \boldsymbol{v^{s}}\rbrace \subset L^{\boldsymbol{p}}\left(\Omega;\boldsymbol{\Lambda^{k}}\right)$ 
 such that $\lbrace \lvert v^{s}_{i} \rvert^{p_{i}}\rbrace$ is equiintegrable and 
\begin{align*}
 v^{s}_{i} \rightharpoonup d\omega_{i} \text{ in } L^{p_{i}}(\Omega, \Lambda^{k_{i}}) 
 \end{align*}
 for all $1 \leq i \leq m$  and 
 \begin{align*}
 \lim_{s \rightarrow \infty} \operatorname*{meas} \Omega_{s} = 0,
\end{align*}
where 
$$ \Omega_{s} := \lbrace x \in \Omega:   v^{s}_{i}(x) \neq d\omega^{s}_{i}(x) \text{ for some } i \in \lbrace 1,\ldots, m \rbrace \rbrace .$$ 
\end{lemma}
\begin{proof}
Since $1 < p_{i} < \infty$ for all $1 \leq i \leq m,$ for every $r$, we find $\boldsymbol{\beta}^{r} \in W^{1, \boldsymbol{p}}(\Omega; \boldsymbol{\Lambda^{k}}),$
such that,
\begin{equation*}
   \left\lbrace \begin{aligned}
                \boldsymbol{d\beta}^{r} = \boldsymbol{d\omega}^{r}  \quad &\text{and} \quad  \boldsymbol{\delta \beta}^{r} = 0 &&\text{ in } \Omega, \\
                \nu\lrcorner \boldsymbol{\beta}^{r} &=  0 &&\text{  on } \partial\Omega,
                \end{aligned} 
                \right. 
                \end{equation*}
and there exists $c_{1} > 0$ such that
$$ \lVert \boldsymbol{\beta}^{r} \rVert_{W^{1,\boldsymbol{p}}} \leq c_{1} \lVert \boldsymbol{d\omega}^{r} \rVert_{L^{\boldsymbol{p}}} .$$
Therefore, up to the extraction of a subsequence which we do not relabel, there exists 
$\boldsymbol{\beta} \in  W^{1,\boldsymbol{p}}\left(\Omega;\boldsymbol{\Lambda^{k -1}}\right)$ such that 
$$ \boldsymbol{\beta}^{r} \rightharpoonup \boldsymbol{\beta}\qquad \text{ in } W^{1,\boldsymbol{p}}\left(\Omega;\boldsymbol{\Lambda^{k -1}}\right).$$
Using a well-known decomposition lemma in calculus of variations (cf. Lemma 2.15 in \cite{MullerFonseca}) to find a subsequence $\lbrace \boldsymbol{\beta}^{s} \rbrace$ and  a sequence 
$\lbrace \boldsymbol{u}^{s} \rbrace \subset W^{1,\boldsymbol{p}}\left(\Omega;\boldsymbol{\Lambda^{k -1}}\right)$ 
such that $\lbrace \lvert \nabla u^{s}_{i} \rvert^{p_{i}}\rbrace$ is equiintegrable for all $1 \leq i \leq m$ and 
\begin{align*}
 \boldsymbol{u}^{s} \rightharpoonup \boldsymbol{\beta} \text{ in } W^{1, \boldsymbol{p}}(\Omega, \boldsymbol{\Lambda^{k-1}}) 
\end{align*}
and $\displaystyle \lim_{\nu \rightarrow \infty} \operatorname*{meas} \Omega^{'}_{s} = 0$
where $\Omega^{'}_{s} = \bigcup\limits_{i=1}^{m} \Omega_{s}^{i} $ with 
$ \Omega_{s}^{i} := \lbrace x \in \Omega:   u^{s}_{i}(x) \neq \beta^{s}_{i}(x) \rbrace \cup 
\lbrace x \in \Omega:   \nabla u^{s}_{i}(x) \neq \nabla\beta^{s}_{i}(x) \rbrace,$ for all $1 \leq i \leq r.$ Setting $\boldsymbol{v}^{s} = \boldsymbol{du}^{s}$ proves the lemma. 
\end{proof}
\begin{remark} (i) In contrast to the classical case, when $k_{i} > 1$ for some $i,$ this lemma does not allow us to replace the sequence 
$\lbrace \boldsymbol{\omega}^{s}\rbrace$ up to a set of small measure.\smallskip

\noindent (ii) The hypothesis of the lemma can be weakened a bit. The conclusion of the lemma still holds if we only require $\boldsymbol{d\omega}^{r} \rightharpoonup 
\boldsymbol{d\omega}$ in $L^{\boldsymbol{p}}(\Omega; \boldsymbol{\Lambda^{k}})$ with the same proof.
\end{remark}

\noindent With lemma \ref{equivalence of quasiconvexity} at hand,  using 
De Giorgi's slicing technique \mbox{\cite{DeGiorgisemicontinuity}} (see also 
    \mbox{\cite{AcerbiFuscosemicontinuity},\cite{Marcelliniapproximationofquasiconvex},\cite{Morrey1966}}) 
as in the proof of its analogue in classical case (cf. Lemma 8.7 in \cite{DCV2}), we can deduce 
the following lemma.

\begin{lemma}\label{semicontinuity lemma}
 Let $\boldsymbol{p} = (p_1,\ldots, p_m)$ where $1\leq p_i \leq \infty$ for all $1 \leq i \leq m. $   Let $ D\subset\mathbb{R}^n$ be a cube parallel to the axes.
 Let $\boldsymbol{\xi} =  (\xi_1, \ldots, \xi_m ) \in \boldsymbol{\Lambda^{k}}$. Let $f:\boldsymbol{\Lambda^k}\rightarrow \mathbb{R}$ be vectorially ext. 
 quasiconvex satisfying the growth condition $\left( \mathcal{C}_{\boldsymbol{p}}\right).$
Let $$ \boldsymbol{\phi^{\nu}} \rightharpoonup 0 \quad \text{ in } W^{d,\boldsymbol{p}}\left(D;\boldsymbol{\Lambda^{k-1}}\right) \quad 
  ( \stackrel{\ast}{\rightharpoonup} \text{ if } p_{i}= \infty ), $$
together with 
$$ \phi^{\nu}_{i} \rightarrow 0 \quad \text{ in } L^{1}\left( D; \Lambda^{k_{i} -1} \right) \text{ if } p_{i} =1.$$
Then
$$\liminf_{\nu \rightarrow \infty} \int_{D} f(\boldsymbol{\xi} + d \boldsymbol{\phi^{\nu}}) \geq  f(\boldsymbol{\xi} ) \operatorname*{meas}(D)  . $$
\end{lemma}
\begin{proof}
 Note that by solving a boundary value problem as in the previous lemma, we can assume 
 $ \phi^{\nu}_{i} \rightharpoonup 0$ in  $W^{1,p_{i}}$  for all $i$ with $1 < p_{i} < \infty. $ By compactness of the embedding, this implies 
 $\phi^{\nu}_{i} \rightarrow 0 $ in $L^{p_{i}}\left(D; \Lambda^{k_{i}-1} \right).$
 If $p_{i} = \infty ,$ then by solving the same boundary value problem
  for some $n < q < \infty,$ we can assume $\phi^{\nu}_{i} \rightharpoonup 0$ in  $W^{1,q} \left(D; \Lambda^{k_{i}-1} \right).$ Compact embedding result then 
  implies $\phi^{\nu}_{i} \rightarrow 0 $ in $L^{\infty}\left(D; \Lambda^{k_{i}-1} \right).$ Thus, we can assume that 
  \begin{align*}
   \boldsymbol{d\phi^{\nu}} &\rightharpoonup 0 \quad \text{ in } L^{\boldsymbol{p}}\left(D;\boldsymbol{\Lambda^{k}}\right) \quad 
  ( \stackrel{\ast}{\rightharpoonup} \text{ if } p_{i}= \infty ), \\
\intertext{and }
\boldsymbol{\phi^{\nu}} &\rightarrow 0 \quad \text{ in } L^{\boldsymbol{p}}\left(D;\boldsymbol{\Lambda^{k-1}}\right).
  \end{align*}
 Now we choose a nested sequence of cubes, 
 each having sides parallel to the axes and each being compactly contained in the next. More precisely, we write $D^{0} \subset D^{1} \subset \ldots \subset D^{\mu} \subset 
 \ldots \subset D^{M} \subset  D,$ where $M \geq 1$ is a positive integer, $ \displaystyle R := \frac{1}{2}\operatorname*{dist} (D^{0}, \partial D)$ 
 and $\displaystyle  \operatorname*{dist} (D^{0}, \partial D^{\mu}) = \frac{\mu}{M}R,\text{ for all } 1\leq \mu \leq M.$
Then we choose $\theta_{\mu} \in C_{c}^{\infty}(D), 1 \leq \mu \leq M,$ such that 
 \begin{align*}
  0 \leq \theta_{\mu} \leq 1,\ \lvert \nabla\theta_{\mu} \rvert \leq \frac{aM}{R},\ \theta_{\mu} = \left\lbrace\begin{aligned}
                                                                                                             &1 \qquad \text{ if } x \in D^{\mu -1} \\
                                                                                                             &0 \qquad \text{ if } x \in D \setminus D^{\mu},
                                                                                                            \end{aligned}\right. \end{align*} 
where $a > 0$ is a constant. We now set $ \boldsymbol{\omega^{\nu}_{\mu}} = \theta_{\mu}\boldsymbol{\phi^{\nu}}
\in W^{d,\boldsymbol{p}}_{T}\left(\Omega;\boldsymbol{\Lambda^{k-1}}\right) $ and use lemma \mbox{\ref{equivalence of quasiconvexity}} to obtain,
\begin{align*}
 \int_{D} f(\boldsymbol{\xi}) &\leq \int_{D} f(\boldsymbol{\xi} + \boldsymbol{d\omega^{\nu}_{\mu}}(x)) \\
 &= \int_{D \setminus D^{\mu}} f(\boldsymbol{\xi}) +  \int_{D^{\mu}\setminus D^{\mu -1}} f(\boldsymbol{\xi} + \boldsymbol{d\omega^{\nu}_{\mu}}(x))
 +  \int_{D^{\mu -1}} f(\boldsymbol{\xi} + \boldsymbol{d\phi^{\nu}}(x)).
\end{align*}
This implies,
$$\int_{D^{\mu}} f(\boldsymbol{\xi}) \leq  \int_{D} f(\boldsymbol{\xi} + \boldsymbol{d\phi^{\nu}}(x))
   - \int_{D\setminus D^{\mu -1}} f(\boldsymbol{\xi} + \boldsymbol{d\phi^{\nu}}(x))+ \int_{D^{\mu}\setminus D^{\mu -1}} f(\boldsymbol{\xi} + \boldsymbol{d\omega^{\nu}_{\mu}}(x)) $$
Using the growth conditions and enlarging the domain of integration to $D\setminus D^{0},$ it is easy to see that the integral over $D\setminus D^{\mu -1}$ can be made 
arbitrarily small by choosing 
$R$ small enough. Growth conditions, bounds for $\theta_{\mu}, \nabla \theta_{\mu}$  and uniform bounds for $\phi_{i}^{\nu}$ in $W^{d,\infty}$ if $p_{i} = \infty$ gives,
\begin{multline*}
 \left\lvert \int_{D^{\mu}\setminus D^{\mu -1}} f(\boldsymbol{\xi} + \boldsymbol{d\omega^{\nu}_{\mu}}(x)) \right\rvert \\ \leq \alpha^{'} \int_{D^{\mu}\setminus D^{\mu -1}} 
 \left( 1 + \sum_{\substack{ i\\p_{i} \neq \infty}} \left( \gamma_{i}\lvert \xi_{i}\rvert^{p_{i}} 
 + \gamma^{'}_{i}\lvert d\phi^{\nu}_{i}\rvert^{p_{i}} + \gamma^{''}_{i} \left(\frac{a M}{R} \right)^{p_{i}}\lvert \phi^{\nu}_{i}\rvert^{p_{i}} \right) \right).
\end{multline*}
Now we sum over $1 \leq \mu \leq M$ and since the sum of the integrals over $D^{\mu}\setminus D^{\mu -1}$ telescopes, we get, after dividing by $M,$
\begin{multline*}
  \int_{D} f(\boldsymbol{\xi} + \boldsymbol{d\phi^{\nu}}(x)) - \left(\frac{ 1 }{M} \sum_{\mu = 1}^{M}\operatorname*{meas}(D^{\mu}) \right)f(\boldsymbol{\xi}) 
  \\ \geq -\varepsilon 
 - \frac{\alpha^{''}}{M}\int_{D^{M}\setminus D^{0}}\left( 1 + \sum_{\substack{ i\\p_{i} \neq \infty}} \left( \gamma^{'}_{i}\lvert 
 d\phi^{\nu}_{i}\rvert^{p_{i}} + \gamma^{''}_{i} \left(\frac{a M}{R} \right)^{p_{i}}\lvert \phi^{\nu}_{i}\rvert^{p_{i}} \right) \right) .
\end{multline*}
We let $\nu \rightarrow \infty.$ Using the fact that 
$\phi^{\nu}_{i} \rightarrow 0$ in $L^{p_{i}},$ choosing $R$ small enough, we get,
\begin{equation*}
  \int_{D} f(\boldsymbol{\xi} + \boldsymbol{d\phi^{\nu}}(x)) - \left(\frac{ 1 }{M} \sum_{\mu = 1}^{M}\operatorname*{meas}(D^{\mu}) \right)f(\boldsymbol{\xi}) 
  \geq -\varepsilon 
 - \frac{\alpha^{'''}}{M}.
\end{equation*}
Since $\operatorname*{meas} (D_{0}) \leq \frac{ 1 }{M} \sum_{\mu = 1}^{M}\operatorname*{meas}(D^{\mu})   \leq \operatorname*{meas} (D) ,$ letting $M \rightarrow \infty$ proves 
the lemma.
\end{proof}
\begin{remark}\label{remarkforstrongconvergence}
  (i) Since the lemma is essentially about changing the boundary values of a sequence up to a set of small measure, we can replace the additional assumption of strong convergence 
 $\phi^{\nu}_{i} \rightarrow 0 \text{ in } L^{1}\left( D; \Lambda^{k_{i} -1} \right)$ if  $p_{i} =1,$ by the assumption that $\phi^{\nu}_{i} \subset 
 W_{T}^{d,1}\left( D; \Lambda^{k_{i} -1} \right)$ for $p_{i} =1, k_{i} >1.$ In that case, we set $ \omega^{\nu}_{\mu, i} = \phi^{\nu}_{i} $ if $p_{i} =1$ and $k_{i} >1$ and 
 $ \omega^{\nu}_{\mu, i} = \theta_{\mu}\phi^{\nu}_{i} $ otherwise. Rest of the proof remains exactly the same as above.\smallskip
 
 (ii) If both $k_{i}=p_{i}=1,$ then the extra assumption of strong convergence is automatically satisfied, thanks to compactness of the embedding.\smallskip 
 
  (iii) The strong convergence assumption in $L^{1}$ or the assumption of the same boundary values, is quite common already in the classical calculus of variations 
  if we weaken the assumption of weak convergence of the gradients, see for example \mbox{\cite{FonsecaLeoniMullerGap}, \cite{FonsecaMullerLscL1}}, 
  also \mbox{\cite{KristensenBV}, \cite{KristensenRindlerBV}}.
\end{remark}

\begin{theorem}\label{semicontinuity without x}
  Let $ 0 \leq r \leq m$ be an integer. 
 $\boldsymbol{p} = (p_1,\ldots, p_m)$ where $1\leq p_i < \infty$ for all $1 \leq i \leq r $ and  $p_{r+1} = \ldots = p_{m} = \infty.$ 
 Let $\Omega\subset\mathbb{R}^n$ be open, bounded, smooth. 
 Let $f:\boldsymbol{\Lambda^k}\rightarrow \mathbb{R}$ be vectorially ext. quasiconvex, satisfying the growth condition $\left( \mathcal{C}_{\boldsymbol{p}}\right)$ . 
Let $$ \boldsymbol{\omega^{\nu}} \rightharpoonup \boldsymbol{\omega} \quad \text{ in } W^{d,\boldsymbol{p}}\left(D;\boldsymbol{\Lambda^{k-1}}\right) \quad 
  ( \stackrel{\ast}{\rightharpoonup} \text{ if } p_{i}= \infty ), $$
together with, 
\begin{align*}
  \text{ if } p_{i} =1,\text{ but } k_{i} \neq 1, \quad \left\lbrace \begin{aligned}
  &\text{either } \omega^{\nu}_{i} \rightarrow \omega_{i} \quad \text{ in } L^{1}\left( D; \Lambda^{k_{i} -1} \right) \\
   &\text{or }\omega^{\nu}_{i} - \omega_{i} \in W^{d,1}_{T}\left( D; \Lambda^{k_{i} -1} \right).
 \end{aligned}\right.
\end{align*}
Then
$$\liminf_{\nu \rightarrow \infty} \int_{\Omega} f( \boldsymbol{d\omega^{\nu}}) \geq  \int_{\Omega} f(\boldsymbol{d\omega} )  . $$
\end{theorem}
\begin{remark}
  The theorem allows $p_{i} =1$ for some (or all) $i,$ with the mentioned additional assumption if $k_{i} > 1$ as well. However, even for $m=1$ and $k=1,$ this is not enough for minimization problems in  $W^{1,1}$, as in well-known in 
 the classical calculus of variations. Since $W^{1,1}$ is non-reflexive, minimizing sequences, even if uniformly bounded in $W^{1,1}$ norm, 
 need not weakly converge to a weak limit in $W^{1,1}.$   
\end{remark}

\begin{proof}
  We need to show that 
 $$ \liminf_{\nu \rightarrow \infty} I(\boldsymbol{\omega^{\nu}}) \geq I(\boldsymbol{\omega}), $$
 for any sequence 
 $$ \boldsymbol{\omega^{\nu}} \rightharpoonup \boldsymbol{\omega} \quad \text{ in } W^{d,\boldsymbol{p}}\left(\Omega;\boldsymbol{\Lambda^{k-1}}\right) \quad 
  ( \stackrel{\ast}{\rightharpoonup} \text{ if } p_{i}= \infty ). $$\bigskip
 
We divide the proof into several steps.

\emph{Step 1} First we show that it is enough to prove the theorem under the additional hypotheses that $\lvert d\omega^{\nu}_{i} \rvert^{p_{j}}$ is equiintegrable for every 
$1 \leq i  \leq r.$ Suppose we have shown the theorem with this additional assumption. Then for any sequence 
$$ \boldsymbol{\omega^{\nu}} \rightharpoonup \boldsymbol{\omega} \quad \text{ in } W^{d,\boldsymbol{p}}\left(\Omega;\boldsymbol{\Lambda^{k-1}}\right), $$
 we first restrict our attention to a subsequence, still denoted by $\lbrace \boldsymbol{\omega^{\nu}} \rbrace$ such that the limit inferior is realized, i.e 
 \begin{equation*}
  L:= \liminf_{\nu \rightarrow \infty} \int_{\Omega}f\left(\boldsymbol{d \omega^{\nu}}(x)\right)\,dx = \lim_{\nu \rightarrow \infty} \int_{\Omega}f\left(\boldsymbol{d \omega^{\nu}}(x)\right)\,dx .
 \end{equation*}

Now we use lemma \ref{decomposition lemma} to find, passing to a subsequence if necessary, a sequence 
$\lbrace v^{\nu}_{i} \rbrace \subset L^{p_{i}}$ such that $\lbrace \lvert v^{\nu}_{i} \rvert^{p_{i}}\rbrace$ is equiintegrable and 
\begin{align*}
 v^{\nu}_{i} \rightharpoonup d\omega_{i} \text{ in } L^{p_{i}}(\Omega, \Lambda^{k_{i}}) \\
 \intertext{ and }
 \lim_{\nu \rightarrow \infty} \operatorname*{meas} \Omega_{\nu} = 0,
\end{align*}
where 
$$ \Omega_{\nu} := \lbrace x \in \Omega:   v^{\nu}_{i}(x) \neq d\omega^{\nu}_{i}(x) \rbrace ,$$ for all $1 \leq i \leq r$ with $p_{i} > 1.$ 
Note also that if $p_{i} = 1,$ we can take
 $v^{\nu}_{i} = d\omega_{i}^{\nu},$  since equiintegrability follows from the weak convergence.  
 
Now, we have, using $\left( \mathcal{C}_{\boldsymbol{p}}\right)$,
\begin{align*}
 \int_{\Omega}f\left(\boldsymbol{d \omega^{\nu}}(x)\right)\,dx \geq \int_{\Omega\setminus \Omega_{\nu}}f\left(v_{1}^{\nu}(x), \ldots, v_{r}^{\nu}(x),d \omega_{r+1}^{\nu}(x), \ldots,
\right. \left. d \omega_{r+1}^{\nu}(x) \right)\,dx \\- \alpha \int_{\Omega_{\nu}} \left( C + \sum_{i=1}^{r} \lvert d\omega^{\nu}_{i}\rvert^{\widetilde{q}_{i}} \right),
\end{align*}
where $C$ is a positive constant, depending on the uniform $L^{\infty}$ bounds of $\lbrace d\omega^{\nu}_{i} \rbrace$ and $\eta_{i}$s in $\left( \mathcal{C}_{\boldsymbol{p}}\right)$, 
for all $r+1 \leq i \leq m$ and $\widetilde{q}_{i} = q_{i},$ as given in $\left( \mathcal{C}_{\boldsymbol{p}}\right)$, if $p_{i} > 1$ and $\widetilde{q}_{i} = 1$ if $p_{i} = 1$ 
for any $1 \leq i \leq m.$

Using $\left( \mathcal{C}_{\boldsymbol{p}}\right)$ again, we obtain,
\begin{align*}
 \int_{\Omega}f\left(\boldsymbol{d \omega^{\nu}}(x)\right)\geq \int_{\Omega}f\left(v_{1}^{\nu}, \ldots, v_{r}^{\nu},d \omega_{r+1}^{\nu}, \ldots,
 d \omega_{r+1}^{\nu} \right) \\ - \alpha \int_{\Omega_{\nu}} \left( C + \sum_{i=1}^{r} \left( \lvert d\omega^{\nu}_{i}\rvert^{\widetilde{q}_{i}} + 
 \lvert v^{\nu}_{i}\rvert^{p_{i}} \right) \right) .
 \end{align*}
 Now we have $ \lim_{\nu \rightarrow \infty} \operatorname*{meas} \Omega_{\nu} = 0$ , $\lbrace \lvert v^{\nu}_{i} \rvert^{p_{i}}\rbrace$ is equiintegrable by construction and 
 $\lbrace \lvert d\omega^{\nu}_{i}\rvert^{\widetilde{q}_{i}} \rbrace $ is equiintegrable 
 since $\widetilde{q}_{i} = q_{i} < p_{i}$ if $p_{i} > 1$ and $\widetilde{q}_{i} = 1$ if $p_{i} = 1.$ Using these facts, we obtain,
 \begin{align*}
  L = \lim_{\nu \rightarrow \infty} \int_{\Omega}f\left(\boldsymbol{d \omega^{\nu}}(x)\right)\,dx \geq \liminf_{\nu \rightarrow \infty}
  \int_{\Omega}f\left(v_{1}^{\nu}, \ldots, v_{r}^{\nu},d \omega_{r+1}^{\nu}, \ldots,
 d \omega_{r+1}^{\nu} \right) \\ \geq \int_{\Omega}f\left(\boldsymbol{d \omega}(x)\right)\,dx,
 \end{align*}
by hypotheses. This proves our claim.\smallskip

\noindent \emph{Step 2} Now by Step 1, we can assume, in addition that  $\lvert d\omega^{\nu}_{i} \rvert^{p_{j}}$ is equiintegrable for every 
$1 \leq i  \leq r.$ Now we approximate $\Omega$ by a union of cubes $D_{s}$ with sides parallel to the axes and whose edge length is $\frac{1}{h}$, where $h$ is an integer.
We denote this union by $H_{h}$ and choose $h$ large enough such that
$$ \operatorname*{meas}(\Omega - H_{h}) \leq \delta \quad \text{ where } H_{h}:= \bigcup D_{s}.$$

Also, we define the average of $d\omega_{i}$ over each of the cubes $D_{s}$ to be,
$$ \xi^{i}_{s} := \frac{1}{\operatorname*{meas}(D_{s})}\int_{D_{s}}d\omega_{i} \in \Lambda^{k_{i}}. $$
Also, let $\boldsymbol{\xi_{s}} := \left( \xi_{s}^{1},\ldots,\xi^{m}_{s}\right)$ and $\displaystyle \boldsymbol{\xi} (x) := \boldsymbol{\xi_{s}}\chi_{D_{s}}(x) $ for every $x \in H_{h}.$ 
Since as the size of the cubes shrink to zero, $d \omega_{i}$ converges to $\xi_{i}$ in $ L^{p_{i}}\left(\Omega;\Lambda^{k_{i}}\right)$ for each $1 \leq i \leq r,$ we obtain, 
by choosing $h$ large enough,
\begin{equation}\label{shrinking cube estimate 11}
 \left( \sum_{s} \int_{D_{s}} \lvert d\omega_{i} -\xi^{i}_{s} \rvert^{p_{i}} \right)^{\frac{1}{p_{i}}} \leq C_{1}\epsilon,
\end{equation}
for every $1 \leq i \leq r.$
Also, by the same argument, we obtain, by choosing $h$ large enough,
\begin{equation}\label{shrinking cube estimate 12}
  \sum_{s} \int_{D_{s}} \lvert d\omega_{i} -\xi^{i}_{s} \rvert  \leq C_{2}\epsilon,
\end{equation}
for every $r+1 \leq i \leq m.$

Now consider
\begin{align*}
 I( \boldsymbol{\omega^{\nu}}) - I(\boldsymbol{\omega}) &= \int_{\Omega} \left[ f\left( \boldsymbol{d\omega^{\nu}} (x) \right) 
 -  f\left( \boldsymbol{d \omega}(x)\right) \right]\,dx \\
  &= I_{1} + I_{2} + I_{3} + I_{4},
\end{align*}
where
\begin{gather*}
 I_{1} := \int_{\Omega - H_{h}} \left[ f\left(\boldsymbol{d \omega^{\nu}}(x)\right) -  f\left(\boldsymbol{d \omega}(x)\right) \right]\,dx , \\
  I_{2} := \sum_{s} \int_{D_{s}} \left[ f\left(\boldsymbol{d\omega} + ( \boldsymbol{d \omega^{\nu}} - \boldsymbol{d\omega}) \right) 
  -  f\left(\boldsymbol{\xi_{s}} + ( \boldsymbol{d \omega^{\nu}} - \boldsymbol{d\omega}) \right)  \right]\,dx , \\
   I_{3} := \sum_{s} \int_{D_{s}}\left[ f\left(\boldsymbol{\xi_{s}} + ( \boldsymbol{d \omega^{\nu}} - \boldsymbol{d\omega}) \right) -  f\left(\boldsymbol{\xi_{s}}\right) \right]\,dx , \\
    I_{4} := \sum_{s} \int_{D_{s}} \left[ f\left(\boldsymbol{\xi_{s}}\right) -  f\left(\boldsymbol{d \omega}\right) \right]\,dx .
\end{gather*}
Now we need to estimate $I_{1}, I_{2}$ and $I_{4}.$ The estimate of $I_{1}$ is similar to the classical case using the 
growth condition $\left( \mathcal{C}_{\boldsymbol{p}}\right)$. We only show the estimate on $I_{2},$ as the estimate of $I_{4}$ can be proved similarly.\smallskip

\emph{Estimation of $I_{2}$:} Since $f$ is vectorially ext. quasiconvex, it is separately convex and since both 
$\left\lbrace d\omega_{i} + ( d \omega^{\nu}_{i} - d\omega_{i}) \right\rbrace$ and 
$ \left\lbrace \xi^{i}_{s} + ( d \omega^{\nu}_{i} - d\omega_{i}) \right\rbrace$ is uniformly bounded in 
$L^{\boldsymbol{\infty}}\left(\Omega;\Lambda^{k_{i}}\right)$ for every $r+1 \leq i \leq m$, using proposition \ref{semicontinuity p lipscitz inequality}, we have,
\begin{align*}
 \lvert I_{2} \rvert \leq  &\sum_{s}  \int_{D_{s}}\sum_{i=1}^{r} \beta_{i}  \left(  1 + \sum_{j=1}^{r} \left( 
 \lvert  d \omega^{\nu}_{j}  \rvert^{\frac{p_{j}}{p'_{i}} }  + 
   \lvert \xi^{j}_{s} + ( d \omega^{\nu}_{j} - d\omega_{j})  \rvert^{\frac{p_{j}}{p'_{i}}} \right)  \right)  \lvert  d\omega_{i} -  \xi^{i}_{s} \rvert \notag 
    \\ &+ \sum_{s} \int_{D_{s}} \sum_{i=r+1}^{m} \beta_{i} \left(  1 + \sum_{j=1}^{r} \left( \lvert  d \omega^{\nu}_{j} \rvert^{p_{j}}    + 
   \lvert \xi^{j}_{s} + ( d \omega^{\nu}_{j} - d\omega_{j}) \rvert^{p_{j}} \right)  \right)  \lvert  d\omega_{i} -  \xi^{i}_{s} \rvert
\end{align*}
The terms in the first sum can be easily estimated by using H\"{o}lder inequality and the estimate \eqref{shrinking cube estimate 11}. Note also that the exponents 
$\frac{p_{j}}{p'_{i}}$ are the precise exponents for this to work. For the second sum, we have, for some positive constants $\widetilde{\beta}_{i}$s, 
\begin{multline*}
 \sum_{s} \int_{D_{s}} \sum_{i=r+1}^{m} \beta_{i}  \left(  1 + \sum_{j=1}^{r} \left( \lvert  d \omega^{\nu}_{j} \rvert^{p_{j}}    + 
   \lvert \xi^{j}_{s} + ( d \omega^{\nu}_{j} - d\omega_{j}) \rvert^{p_{j}} \right)  \right)  \lvert  d\omega_{i} -  \xi^{i}_{s} \rvert \\
 \leq \sum_{s} \int_{D_{s}} \sum_{i=r+1}^{m} \widetilde{\beta}_{i} \left(  1 + \sum_{j=1}^{r} \left( \lvert d \omega^{\nu}_{j}\rvert^{p_{j}}    + 
   \lvert d\omega_{j} - \xi^{j}_{s} \rvert^{p_{j}} \right)  \right)  \lvert  d\omega_{i} -  \xi^{i}_{s} \rvert.
\end{multline*}

Now the terms of the form 
$$ \sum_{s} \int_{D_{s}}\widetilde{\beta}_{i} \lvert  d\omega_{i} -  \xi^{i}_{s} \rvert $$
can be easily estimated using estimate \eqref{shrinking cube estimate 12}. For the other terms, for any $i,j$, $r+1 \leq i \leq m$ and $1 \leq j \leq r,$ we have, 
\begin{align}
 \sum_{s} \int_{D_{s}} \widetilde{\beta}_{i} \lvert d\omega_{j} - \xi^{j}_{s} \rvert^{p_{j}}\lvert  d\omega_{i} -  \xi^{i}_{s} \rvert 
 \leq 2 \widetilde{\beta}_{i} \lVert d\omega_{i} \rVert_{L^{\infty}(\Omega)} \sum_{s} \int_{D_{s}} \lvert d\omega_{j} - \xi^{j}_{s} \rvert^{p_{j}}.
\end{align}
Using the estimate \eqref{shrinking cube estimate 11}, these terms can be made as small as we please by choosing $h$ large enough. Now we estimate the terms of the type  
$$ \sum_{s} \int_{D_{s}}  \widetilde{\beta}_{i} \lvert d \omega^{\nu}_{j}\rvert^{p_{j}} \lvert  d\omega_{i} -  \xi^{i}_{s} \rvert .$$
Since $\lbrace \lvert d \omega^{\nu}_{j} \rvert^{p_{j}}\rbrace$ is uniformly bounded in $L^1$ and is equiintegrable, we know,
\begin{equation*}
 \lim_{M \rightarrow \infty} \sup_{\nu} \int\limits_{\Omega \cap \lbrace \lvert d \omega^{\nu}_{j} \rvert^{p_{j}} > M \rbrace} \lvert d \omega^{\nu}_{j} \rvert^{p_{j}} = 0. 
\end{equation*}

This implies, for any $\epsilon > 0,$ there exists $M = M(\epsilon)$ such that 

\begin{equation*}
  \int\limits_{\Omega \cap \lbrace \lvert d \omega^{\nu}_{j} \rvert^{p_{j}} > 
  M \rbrace} \lvert d \omega^{\nu}_{j} \rvert^{p_{j}} < \frac{\epsilon}{2\widetilde{\beta}_{i}\lVert d\omega_{i} \rVert_{L^{\infty}(\Omega)}}  \text{ for all } \nu.
\end{equation*}

Thus, we have, for any  $i,j,$
$r+1 \leq i \leq m$ and $1 \leq j \leq r,$
\begin{align*}
 &\sum_{s} \int_{D_{s}}  \widetilde{\beta}_{i} \lvert d \omega^{\nu}_{j}\rvert^{p_{j}} \lvert  d\omega_{i} -  \xi^{i}_{s} \rvert \\
 &= \int\limits_{H_{h} \cap \lbrace \lvert d \omega^{\nu}_{j} \rvert^{p_{j}} > M \rbrace}  
 \widetilde{\beta}_{i} \lvert d \omega^{\nu}_{j}\rvert^{p_{j}} \lvert  d\omega_{i} -  \xi^{i}_{s} \rvert
  + \int\limits_{H_{h} \cap \lbrace \lvert d \omega^{\nu}_{j} \rvert^{p_{j}} \leq M \rbrace}  
 \widetilde{\beta}_{i} \lvert d \omega^{\nu}_{j}\rvert^{p_{j}} \lvert  d\omega_{i} -  \xi^{i}_{s} \rvert \\
 &\leq \epsilon + \widetilde{\beta}_{i} M \sum_{s} \int_{D_{s}}\lvert  d\omega_{i} -  \xi^{i}_{s} \rvert .
\end{align*}
Estimate \eqref{shrinking cube estimate 12} concludes the argument. \smallskip

 Using all the estimates and
taking the limit $\nu \rightarrow \infty $, we obtain,
\begin{multline*}
  \liminf_{\nu \rightarrow \infty}  I( \boldsymbol{\omega^{\nu}}) - I(\boldsymbol{\omega}) \geq -(C_{I_{1}}+C_{I_{3}}+C_{I_{4}})\epsilon \\ + 
\sum_{s} \liminf_{\nu \rightarrow \infty}  \int_{D_{s}}\left[ f\left(\boldsymbol{\xi_{s}} + ( \boldsymbol{d \omega^{\nu}} - \boldsymbol{d\omega}) \right) -  f\left(\boldsymbol{\xi_{s}}\right) \right]\,dx .
\end{multline*}

Since $$\boldsymbol{d \omega^{\nu}} - \boldsymbol{d\omega} 
\rightharpoonup 0 \quad \text{ in } W^{d,\boldsymbol{p}}\left(D_{s};\boldsymbol{\Lambda^{k-1}}\right) $$
and either 
$$\omega^{\nu}_{i} \rightarrow \omega_{i} \quad \text{ in } L^{1}\left( D; \Lambda^{k_{i} -1} \right) \qquad \text{or } \quad 
\omega^{\nu}_{i} - \omega_{i} \in W^{d,1}_{T}\left( D; \Lambda^{k_{i} -1} \right),$$  if $p_{i} =1,$ but $k_{i} \neq 1,$
for every $s$, using lemma \ref{semicontinuity lemma}, remark \ref{remarkforstrongconvergence}(i)
and the fact that $\epsilon$ is arbitrary, we have finished the proof of the theorem.
\end{proof}
\subsection{Lower semicontinuity for general quasiconvex functions}

We first show that the explicit dependence on $x$, but no explicit dependence on $\boldsymbol{\omega}$ for a vectorially ext. quasiconvex functions can be handled in the standard way.
We start by defining the growth conditions that we need for this case.
\begin{definition}[Growth conditions II]
 Let $\Omega\subset\mathbb{R}^n$ be open, bounded.
  Let $f:\Omega\times\boldsymbol{\Lambda^k}\rightarrow \mathbb{R}$ be a Carath\'eodory function. \smallskip
  
  \noindent $f$ is said to be of growth $\left(\mathcal{C}^{x}_{\boldsymbol{p}}\right)$, if , for almost every $x \in \Omega$ and
  for every $\boldsymbol{\xi} =  (\xi_1, \ldots, \xi_m ) \in \boldsymbol{\Lambda^{k}}$,  $f$ satisfies, 
  \begin{align*}
   -\beta(x) - \sum_{i=1}^{m}  G^{l}_{i}( \xi_{i} )  \leq f(x, \boldsymbol{\xi}) \leq \beta(x) +  \sum_{i=1}^{m} G^{u}_{i}(\xi_i) ,\tag{ $\mathcal{C}^{x}_{\boldsymbol{p}} $}
  \end{align*}
where $\beta \in L^{1}(\Omega)$ is nonnegative and the functions $G^{l}_{i}$s in the lower bound and the functions $G^{u}_{i}$s in the upper bound has the following form:
\begin{itemize}
 \item If $p_i =1$, then,
                   \begin{align*}
                    G^{l}_{i}(\xi_i) = G^{u}_{i}(\xi_i) = \alpha_{i} \lvert \xi_{i} \rvert \qquad \qquad \text{ for some constant } \alpha_i \geq 0. 
                    \end{align*}
\item If $1 < p_i < \infty$, then,
\begin{align*}
 G^{l}_{i}(\xi_i) = \alpha_{i}\lvert \xi_{i} \rvert^{q_i}   \quad \text{ and } \quad G^{u}_{i}(\xi_i) = g_{i}(x)\lvert \xi_{i} \rvert^{p_i} , 
\end{align*}
for some $ 1 \leq q_i < p_i $ and for some constant $ \alpha_i \geq 0$ and some non-negative measurable function $g_{i}.$ 
\item If $p_i = \infty$, then, 
\begin{align*}
 G^{l}_{i}(\xi_i) = G^{u}_{i}(\xi_i) = \eta_{i}\left( \lvert \xi_{i} \rvert \right)  . 
\end{align*}
for some nonnegative, continuous, increasing function  $\eta_i$. 
\end{itemize}
\end{definition}
Under these growth conditions,  we can prove the semicontinuity result for functionals with explicit dependence on $x$. With theorem \ref{semicontinuity without x} in hand, 
the proof is very similar to classical way to handle measurable dependence on $x$ in semicontinuity theorems (cf. theorem 8.8 and theorem 8.11 in \cite{DCV2}). 
\begin{theorem}[Sufficient condition]\label{semicontinuity with x dependence}
   Let $ 0 \leq r \leq m$ be an integer. 
 $\boldsymbol{p} = (p_1,\ldots, p_m)$ where $1\leq p_i < \infty$ for all $1 \leq i \leq r $ and  $p_{r+1} = \ldots = p_{m} = \infty.$ 
 Let $\Omega\subset\mathbb{R}^n$ be open, bounded, smooth. 
 Let $f:\Omega\times\boldsymbol{\Lambda^k}\rightarrow \mathbb{R}$ be a Carath\'eodory function, satisfying the growth condition $\left( \mathcal{C}^{x}_{\boldsymbol{p}}\right)$ 
 and $\boldsymbol{\xi} \mapsto f(x, \boldsymbol{\xi})$ is vectorially ext. quasiconvex for a.e $x \in \Omega.$ 
 Let $$ \boldsymbol{\omega^{\nu}} \rightharpoonup \boldsymbol{\omega} \quad \text{ in } W^{d,\boldsymbol{p}}\left(D;\boldsymbol{\Lambda^{k-1}}\right) \quad 
  ( \stackrel{\ast}{\rightharpoonup} \text{ if } p_{i}= \infty ), $$
 together with, 
 \begin{align*}
  \text{ if } p_{i} =1,\text{ but } k_{i} \neq 1, \quad \left\lbrace \begin{aligned}
  &\text{either } \omega^{\nu}_{i} \rightarrow \omega_{i} \quad \text{ in } L^{1}\left( D; \Lambda^{k_{i} -1} \right) \\
   &\text{or }\omega^{\nu}_{i} - \omega_{i} \in W^{d,1}_{T}\left( D; \Lambda^{k_{i} -1} \right).
 \end{aligned}\right.
\end{align*}
Then
$$\liminf_{\nu \rightarrow \infty} \int_{\Omega} f( x, \boldsymbol{d\omega^{\nu}}) \geq  \int_{\Omega} f(x, \boldsymbol{d\omega} )  . $$
\end{theorem}
\begin{proof}
The argument works in two stages. First we show that to prove the theorem,
\begin{itemize}
 \item[(A1)] We can assume $f$ satisfies a slightly more restrictive growth condition, namely, 
for almost every $x \in \Omega$ and for every 
 $\boldsymbol{\xi} \in \boldsymbol{\Lambda^{k}},$ 
 \begin{equation}\label{simpler growth}
 -\sum_{ \substack{ i\\p_{i} =1}} \alpha_{i}\lvert \xi_{i}\rvert   \leq f(x, \boldsymbol{\xi}) \leq \beta(x) +  \sum_{i=1}^{r} \alpha_{i}\lvert \xi_{i} \rvert^{p_i} + \sum_{i=r+1}^{m} \eta_{i}\left( \lvert \xi_{i} \rvert \right) 
  ,\tag{ $\mathcal{C}^{x'}_{\boldsymbol{p}} $}
 \end{equation}
for some nonnegative $\beta \in L^{1}(\Omega)$, where  $ \alpha_i \geq 0$ for all $1 \leq i \leq r$ are constants and  $\eta_i$s are  some nonnegative, continuous, 
increasing function for each $r+1 \leq i \leq m.$
\item[(A2)] We can restrict our attention to sequences $\boldsymbol{\omega^{\nu}} \rightharpoonup \boldsymbol{\omega} \text{ in } W^{d,\boldsymbol{p}}\left(\Omega;\boldsymbol{\Lambda^{k-1}}\right)$ 
with the property that $\lbrace \lvert d\omega^{\nu}_{i} \rvert^{p_{i}} \rbrace$ is equiintegrable for all $1 \leq i \leq r.$
\item[(A3)] We can assume $\Omega \subset \mathbb{R}^{n}$ is an open cube with sides parallel to 
axes. 
\end{itemize}
To show (A1), first note that for a sequence $\boldsymbol{\omega^{\nu}} \rightharpoonup \boldsymbol{\omega} \text{ in } W^{d,\boldsymbol{p}}\left(\Omega;\boldsymbol{\Lambda^{k-1}}\right)$, 
there exist constants $\gamma_{i} > 0$ such that $\lVert d\omega^{\nu}_{i} \rVert_{L^{\infty}} \leq \gamma_{i} \text { for every } r+1 \leq i \leq m.$  Also, if $1 \leq q_{i} < p_{i},$ then for every $\varepsilon > 0,$ there exists 
a constant $k_{i} = k_{i}(\varepsilon) > 0$ such that 
 $ \varepsilon \lvert \xi_{i}\rvert^{p_{i}} + k_{i} \leq \alpha_{i}\lvert \xi_{i}\rvert^{q_{i}} \text{ for all } \xi_{i} \in \Lambda^{k_{i}}.$ 
 Set $\displaystyle k := \sum_{\substack{i\\ 1 <  p_{i} < \infty}} k_{i} + \sum\limits_{i =r+1}^{m} \eta_{i} \left(\gamma_{i} \right).$ and define 
  \begin{align*}
  f_{\varepsilon} (x, \boldsymbol{\xi}) = f(x, \boldsymbol{\xi}) + \beta (x) + \varepsilon \sum_{\substack{i\\ 1 < 
  p_{i} < \infty}} \lvert \xi \rvert^{p_{i}} + k.
 \end{align*}
It is easy to see that if $f$ satisfies $\left( \mathcal{C}^{x}_{\boldsymbol{p}}\right)$, then $f_{\varepsilon}$ satisfies, 
\begin{equation}\label{intermediate growth}
 -\sum_{ \substack{ i\\p_{i} =1}} \alpha_{i}\lvert \xi_{i}\rvert   \leq f(x, \boldsymbol{\xi}) \leq \beta(x) +  \sum_{\substack{i \\ p_{i} =1}} \alpha_{i}\lvert \xi_{i} \rvert 
 + \sum_{\substack{i \\ 1 < p_{i} < \infty }} g_{i}(x)\lvert \xi_{i} \rvert + \sum_{i=r+1}^{m} \eta_{i}\left( \lvert \xi_{i} \rvert \right) 
  .\tag{ $\mathcal{C}^{x''}_{\boldsymbol{p}} $}
 \end{equation}
 $f_{\varepsilon}$ is clearly vectorially ext. quasiconvex and letting $\varepsilon \rightarrow 0,$ we can deduce the semicontinuity result for $f$, along the sequence  
 $\boldsymbol{\omega^{\nu}}$, from the one for  $f_{\varepsilon}.$ This shows that we can replace the conditions $\left( \mathcal{C}^{x}_{\boldsymbol{p}}\right)$ by 
 $(\mathcal{C}^{x''}_{\boldsymbol{p}} )$. To prove (A1), it only remains to show that we can replace the functions $g_{i}(x)$ with constants. We define, for 
 every natural number $\mu,$
\begin{align*}
 \phi^{\mu} (x) := \left\lbrace \begin{aligned}
                                &1  &&\text{ if } \max\limits_{\substack{i \\ 1 < p_{i} < \infty}} g_{i}(x) \leq \mu \\ 
                                &\frac{\mu}{\max\limits_{\substack{i \\ 1 < p_{i} < \infty}}\left[ g_{i}(x)\right] } &&\text{ if  otherwise }.
                               \end{aligned}\right. 
\end{align*}
Setting $ f_{\mu}(x, \boldsymbol{\xi}) := \phi^{\mu} (x) f(x, \boldsymbol{\xi}) ,$ we see that $f_{\mu}$ satisfies $\mathcal{C}^{x'}_{\boldsymbol{p}}$ for every $\mu$ and 
$ f(x, \boldsymbol{\xi})  = \sup_{\mu} f_{\mu}(x, \boldsymbol{\xi}) = \lim_{\mu \rightarrow \infty} f_{\mu}(x, \boldsymbol{\xi}). $ Thus, semicontinuity result for $f$ follows 
from that of $f_{\mu}.$ This proves (A1). Proceeding as in Step 1 of the proof of Theorem \mbox{\ref{semicontinuity without x}} above, we prove (A2). (A3) is shown by approximating 
$\Omega$ from the inside by a finite union of disjoint open cubes with sides parallel to axes, up to a set of small measure and using equiintegrability.\smallskip 

Next we show the theorem under the additional assumptions (A1),(A2),(A3). The strategy is standard. We freeze the points and then use Theorem 
\ref{semicontinuity without x}.
\smallskip

For any given $\varepsilon >0,$ for every $1 \leq i \leq r,$ there exist constants $M^{i}_{\varepsilon} \geq 1$, 
independent of $\nu$, such that the sets $K^{i}_{\varepsilon,\nu} := \left\lbrace  x \in \Omega: \lvert d\omega^{\nu}_{i} \rvert^{p_{i}} \text{ or } 
\lvert d\omega_{i} \rvert^{p_{i}} > M^{i}_{\varepsilon} \right\rbrace, $ satisfy $\operatorname*{meas} \left( K^{i}_{\varepsilon,\nu} \right) < \frac{\varepsilon}{r}, $ 
for every $\nu.$ We set $\Omega_{\varepsilon}:= \Omega \setminus \bigcup\limits_{i=1}^{r}K^{i}_{\varepsilon,\nu}.$ Also, 
for every $ r+1 \leq i \leq m,$ i.e there exist constants $\gamma_{i} > 0$ such that
$ \lVert d\omega^{\nu}_{i} \rVert_{L^{\infty}} \leq \gamma_{i}$ for all $\nu. $ We define $\displaystyle  k:= \sum_{i=r+1}^{m} \eta_{i}(\gamma_{i}) $ and since 
$\beta \in L^{1}(\Omega)$ and nonnegative, given any $\varepsilon > 0,$ we can find $M^{\beta}_{\varepsilon} \leq 1$ such that 
$ \operatorname*{meas} (\Omega\setminus E_{\varepsilon}) \leq \frac{\varepsilon}{k}$ and $\int_{\Omega\setminus E_{\varepsilon}} \beta(x)dx < \varepsilon, $ where 
$ E_{\varepsilon} := \lbrace x \in \Omega : \beta(x) \leq M^{\beta}_{\varepsilon} \rbrace .$ Now by the Scorza-Dragoni theorem (cf. theorem 3.8 in \mbox{\cite{DCV2}}), 
we find a compact set $K_{\varepsilon} \subset \Omega_{\varepsilon}$ with 
$ \operatorname*{meas} (\Omega_{\varepsilon}\setminus K_{\varepsilon}) < \varepsilon $ such that 
$f: K_{\varepsilon} \times S_{\varepsilon} \rightarrow \mathbb{R}$ is continuous, where 
$$  S_{\varepsilon}: = \lbrace \boldsymbol{\xi} \in \boldsymbol{\Lambda^{k}} : \lvert \xi\rvert^{p_{i}} \leq M^{i}_{\varepsilon} \text{ for all } 1 \leq i \leq r,\ 
\lvert \xi\rvert \leq \gamma_{i} \text{ for all } r+1 \leq i \leq m\rbrace .$$
Now we subdivide $\Omega$ into a finite union of cubes $D_{s}$ of side length $\frac{1}{h}$ such that 
$ \operatorname*{meas} \left(\Omega \setminus \bigcup\limits_{s} D_{s}\right) = 0.$ Fix $x_{s} \in D_{s}$ for all $s.$ Now using the uniform continuity of $f$ on the sets
$E_{\varepsilon}\cap K_{\varepsilon}\cap D_{s},$ the lower bound 
and the upper bound, respectively, in (A1) and choosing $h$ large enough,  we can find the estimates 
\begin{align*}
  \int_{\Omega}f\left(x,\boldsymbol{d \omega^{\nu}}\right) &\geq \sum\limits_{s} \int_{D_{s}}f\left(x_{s},\boldsymbol{d \omega^{\nu}}\right)
  - R_{1}\left(\varepsilon \right) ,\\
  \sum\limits_{s} \int_{D_{s}}f\left(x_{s},\boldsymbol{d \omega}\right) &\geq \int_{\Omega}f\left(x,\boldsymbol{d \omega}\right) 
- R_{2}\left(\varepsilon \right),
\end{align*}
where $R_{1}\left(\varepsilon \right),R_{2}\left(\varepsilon \right) \rightarrow 0$ as $\varepsilon \rightarrow 0.$ In view of theorem \mbox{\ref{semicontinuity without x}}, this 
concludes the proof.
\end{proof}

As was pointed out to the author by Kristensen (private communication), it is also possible to give a different proof of both theorem 
\ref{semicontinuity without x} and theorem \ref{semicontinuity with x dependence}, utilizing the blow-up argument of Fonseca-M\"{u}ller \cite{MullerFonseca}.
\subsection{Failure of semicontinuity in $W^{d,p}$ for general functional}
 Vectorial ext. quasiconvexity of the map $\boldsymbol{\xi} \mapsto f(x, \boldsymbol{\omega} , \boldsymbol{\xi})$, along with usual growth conditions,  
 is not sufficient for weak lower semicontinuity in 
 $W^{d,\boldsymbol{p}}$ of functionals 
 with explicit dependence on $\boldsymbol{\omega},$ i.e for functionals of the form,
 $$\int_{\Omega}f\left(x,\boldsymbol{\omega}, \boldsymbol{d \omega}\right)\,dx .$$
For example, even when $m=1$, for $ k \geq 2$, we have the following.
\begin{proposition}[Counterexample to semicontinuity]\label{failure of semicontinuity}
Let  $n \geq 2.$ Also let $2 \leq k \leq n$, $1 \leq  p < \infty $ and let $\Omega \subset\mathbb{R}^n.$ Let
$$ I(\omega):=\frac{1}{p}\int_{\Omega}\lvert d\omega \rvert^{p} -\frac{1}{p}\int_{\Omega} \lvert \omega \rvert^{p},\text{ for all }\omega\in 
W^{d,p}\left(\Omega;\Lambda^{k-1}\right). $$
Then $I$ is not weakly lower semicontinuous in $W^{d,p}\left(\Omega;\Lambda^{k-1}\right).$
\end{proposition}
\begin{proof}
Consider a sequence of exact forms $\lbrace d\theta_{\nu} \rbrace \subset L^{p}\left(\Omega;\Lambda^{k-1}\right)$ such that 
$$d\theta_{\nu} \rightharpoonup d\theta \text{ in } L^{p}\left(\Omega;\Lambda^{k-1}\right) \text{ but } 
d\theta_{\nu} \not\rightarrow d\theta \text{ in } L^{p}\left(\Omega;\Lambda^{k-1}\right),$$
for some $d\theta \in L^{p}\left(\Omega;\Lambda^{k-1}\right)$. Note that finding such a sequence is impossible if $k=1$ and always possible for $2 \leq k \leq n.$
But, then we have, 
\begin{align*}
 \liminf_{\nu \rightarrow \infty} I(d\theta_{\nu}) &= \liminf_{\nu \rightarrow \infty} \left(  -\frac{1}{p}\int_{\Omega} \lvert d\theta_{\nu} \rvert^{p} \right)   
= - \frac{1}{p} \limsup_{\nu \rightarrow \infty}   \int_{\Omega} \lvert d\theta_{\nu} \rvert^{p} \\ 
 &\leq  - \frac{1}{p} \liminf_{\nu \rightarrow \infty}   \int_{\Omega} \lvert d\theta_{\nu} \rvert^{p} \leq -\frac{1}{p}\int_{\Omega} \lvert d\theta \rvert^{p} = I(d\theta).
\end{align*}
But if $I$ is weakly lower semicontinuous, this implies $ \displaystyle \liminf_{\nu \rightarrow \infty} I(d\theta_{\nu}) =  I(d\theta).$
But this is impossible since that would imply,
$$ \limsup_{\nu \rightarrow \infty}   \lVert d\theta_{\nu} \rVert^{p}_{L^{p}} = \liminf_{\nu \rightarrow \infty}   \lVert d\theta_{\nu} \rVert^{p}_{L^{p}} 
= \lim_{\nu \rightarrow \infty}   \lVert d\theta_{\nu} \rVert^{p}_{L^{p}} = \lVert d\theta \rVert^{p}_{L^{p}}. $$
Since $d\theta_{\nu} \rightharpoonup d\theta$ in $L^{p}$, this implies the strong convergence in $L^{p}$, which contradicts the fact that 
$d\theta_{\nu} \not\rightarrow d\theta \text{ in } L^{p}\left(\Omega;\Lambda^{k-1}\right).$
\end{proof}\smallskip 

\noindent However, if $k_{i} = 1$ for all $1 \leq i \leq m,$ the functional $\int_{\Omega}f\left(x,\boldsymbol{\omega}, \boldsymbol{d \omega}\right)\,dx $ is weakly lower 
semicontinuous in $W^{d,\boldsymbol{p}},$ precisely because in this case $W^{d,\boldsymbol{p}}$ and $W^{1,\boldsymbol{p}}$ are the same space. 
Indeed, it is possible to show the more general result that the functional $\int_{\Omega}f\left(x,\boldsymbol{\omega}, \boldsymbol{d \omega}(x)\right)\,dx $ is always weakly lower 
semicontinuous in $W^{1,\boldsymbol{p}}$ with appropriate growth conditions on $f.$ 

\subsection{Semicontinuity in $W^{1,\boldsymbol{p}}$ for general functional}
We first define the appropriate growth conditions in this setting.
\begin{definition}[Growth condition III]
 Let $\Omega\subset\mathbb{R}^n$ be open, bounded.
  Let $f:\Omega\times\boldsymbol{\Lambda^{k-1}}\times\boldsymbol{\Lambda^k}\rightarrow \mathbb{R}$ be a Carath\'eodory function. \smallskip
  
  \noindent $f$ is said to be of growth $\left(\mathcal{C}^{x,u}_{\boldsymbol{p}}\right)$, if , for almost every $x \in \Omega$ and
  for every $( \boldsymbol{u}, \boldsymbol{\xi} )  \in \boldsymbol{\Lambda^{k-1}}\times \boldsymbol{\Lambda^{k}}$,  $f$ satisfies, 
  \begin{align*}
   -\beta(x) - \sum_{i=1}^{m}  G^{l}_{i}( u_{i}, \xi_{i} )  \leq f(x, \boldsymbol{u},\boldsymbol{\xi}) \leq \beta(x) +  \sum_{i=1}^{m} G^{u}_{i}( u_{i}, \xi_i) ,\tag{ $\mathcal{C}^{x,u}_{\boldsymbol{p}} $}
  \end{align*}
where $\beta \in L^{1}(\Omega)$ is nonnegative and the functions $G^{l}_{i}$s in the lower bound and the functions $G^{u}_{i}$s in the upper bound has the following form:
\begin{itemize}
 \item If $p_i =1$, then,
                   \begin{align*}
                    G^{l}_{i}(u_{i}, \xi_i) = G^{u}_{i}(u_{i}, \xi_i) = \alpha_{i} \lvert \xi_{i} \rvert \qquad \qquad \text{ for some constant } \alpha_i \geq 0. 
                    \end{align*}
\item If $1 < p_i < \infty$, then,
\begin{align*}
 G^{l}_{i}(u_{i}, \xi_i) = \alpha_{i}\left( \lvert \xi_{i} \rvert^{q_i} + \lvert u_{i} \rvert^{r_i}\right)  \quad \text{ and } \quad 
 G^{u}_{i}(u_{i},\xi_i) = g_{i}(x, u_{i})\lvert \xi_{i} \rvert^{p_i} , 
\end{align*}
for some $ 1 \leq q_i < p_i ,$ $ 1 \leq r_{i} < np_{i}/(n -p_{i})$ if $p_{i} < n$ and $1 \leq r_{i} < \infty$ if $p_{i} \geq n$, $g_{i}$ is a nonnegative Carath\'eodory function                     
and for some constant $ \alpha_i \geq 0.$ 
\item If $p_i = \infty$, then, 
\begin{align*}
 G^{l}_{i}(u_{i}, \xi_i) = G^{u}_{i}(u_{i}, \xi_i) = \eta_{i}\left( \lvert u_{i} \rvert, \lvert \xi_{i} \rvert \right)  . 
\end{align*}
for some nonnegative, continuous, increasing (in each argument) function  $\eta_i$. 
\end{itemize}
\end{definition}

\noindent With these growth conditions on $f$, it is possible to show that the functional $\int_{\Omega}f\left(x,\boldsymbol{\omega}, \boldsymbol{d \omega}(x)\right)\,dx $ 
is always weakly lower semicontinuous in $W^{1,\boldsymbol{p}}.$ The proof is very similar to the proof of Theorem \ref{semicontinuity with x dependence}. In this case too, it is possible to derive all 
the necessary estimates after freezing both $x$ and $\boldsymbol{\omega}.$ 
Some modifications are required to handle the explicit dependence 
on $\boldsymbol{\omega},$ but these modifications essentially use the Sobolev embedding and is quite standard 
(see theorem 8.8 and theorem 8.11 in \cite{DCV2} for the classical case).  
We state the theorem below and omit the proof.

\begin{theorem}\label{semicontinuity in W1p for general functional}
  Let $\Omega\subset\mathbb{R}^n$ be open, bounded, smooth. 
 Let $f:\Omega\times\boldsymbol{\Lambda^{k-1}}\times\boldsymbol{\Lambda^k}\rightarrow \mathbb{R}$ be a Carath\'eodory function, satisfying the growth condition 
 $\left( \mathcal{C}^{x,u}_{\boldsymbol{p}}\right)$ 
 and $\boldsymbol{\xi} \mapsto f(x, \boldsymbol{u}, \boldsymbol{\xi})$ is vectorially ext. quasiconvex for a.e $x \in \Omega$ and for 
 every $\boldsymbol{u} \in \boldsymbol{\Lambda^{k-1}}.$ 
 Let $I:W^{1,\boldsymbol{p}}\left(\Omega;\boldsymbol{\Lambda^{k-1}}
 \right)\rightarrow \mathbb{R}$ defined by
$$
I(\boldsymbol{\omega}):=\int_{\Omega}f\left(x,\boldsymbol{\omega}, \boldsymbol{d \omega}\right)\,dx,\text{ for all }\boldsymbol{\omega}\in W^{1,\boldsymbol{p}}\left(\Omega;\boldsymbol{
\Lambda^{k-1}}\right).
$$
Then $I$ is  weakly lower semicontinuous in $W^{1,\boldsymbol{p}}\left(\Omega;\boldsymbol{\Lambda^{k-1}}\right)$ (weakly $\ast$ in $i$-th factor if $p_{i} = \infty$).  
\end{theorem}
\begin{remark}
 In the special case when $k_{i}=1$ for all $1 \leq i \leq m,$ this theorem recovers the classical result with the improvement 
 that the $p_{i}$s are allowed to be different from one another. If we take, $p_{i} = p$ for every $1 \leq i \leq m,$ as well, then we obtain
precisely the classical results,  i.e theorem 8.8 or theorem 8.11 in \cite{DCV2}, depending on whether $p=\infty$ or $1 \leq p < \infty.$ 
\end{remark}

\section{Weak Continuity}\label{weakcontinuity}
We now turn our attention to characterizing all sequentially weakly continuous functions in $W^{d,\boldsymbol{p}}(\Omega; \Lambda^{\boldsymbol{k-1}}).$

\begin{definition}[Weak continuity]
Let $\Omega\subset\mathbb{R}^n$ be open and 
let $ \displaystyle f:\boldsymbol{\Lambda^k}\rightarrow\mathbb{R}$ be continuous. We say that $f$ is weakly continuous on 
$\displaystyle W^{d,\boldsymbol{p}}\left(\Omega;\Lambda^{\boldsymbol{k -1}}\right)$, if for every sequence
$\displaystyle \left\lbrace \boldsymbol{\omega}^{\nu} \right\rbrace_{\nu=1}^{\infty} \subset  W^{d,\boldsymbol{p}}\left(\Omega;\Lambda^{\boldsymbol{k -1}}\right) $ satisfying
$\displaystyle  \boldsymbol{\omega}^{\nu} \rightharpoonup \boldsymbol{\omega} \text{ in } W^{d,\boldsymbol{p}}\left(\Omega;\Lambda^{\boldsymbol{k -1}}\right)$ 
for some $\displaystyle \boldsymbol{\omega} \in W^{d,\boldsymbol{p}}\left(\Omega;\Lambda^{\boldsymbol{k -1}}\right),$
we have
$$ f\left( \boldsymbol{d \omega}^{\nu}\right) \rightharpoonup f\left( \boldsymbol{d \omega} \right) \text{ in }\mathcal{D}'(\Omega).$$
\end{definition}
\subsection{Necessary condition}
\begin{theorem}[Necessary condition]\label{weakcontnecessary}
Let $\Omega\subset\mathbb{R}^n$ be open, bounded and let 
$f:\boldsymbol{\Lambda^{k}}\rightarrow\mathbb{R}$ 
be weakly continuous on $W^{d,\boldsymbol{\infty}}\left(\Omega;\boldsymbol{ \Lambda^{k} } \right).$ 
Then, $f$ is vectorially ext. one affine, and hence, is of the form
\begin{equation}\label{21.4.2015.5}
f(\boldsymbol{\xi})= \sum_{ \substack{ \boldsymbol{\alpha},  \\ 0 \leq \lvert \boldsymbol{k\alpha} \rvert \leq n } }\left\langle c_{\boldsymbol{\alpha}} %
;\boldsymbol{\xi^{\alpha}}\right\rangle \text{ for all }\boldsymbol{\xi}\in \boldsymbol{\Lambda^{k}},
\end{equation}
where $ c_{ \boldsymbol{\alpha}}\in \Lambda^{ \lvert \boldsymbol{k\alpha } \rvert }(\mathbb{R}^n),$ for 
every $\boldsymbol{\alpha}$ with  $0 \leq \lvert \boldsymbol{k\alpha } \rvert \leq n$.
\end{theorem}
\begin{remark}
 As in remark \ref{remarkinftynecessaryforp}, $f$ being vectorially ext. one affine is a necessary condition for 
 weak continuity in $W^{d,\boldsymbol{p}}\left(\Omega;\boldsymbol{\Lambda^{k-1}}\right)$ as well.
\end{remark}
\begin{proof}
 Since $f$ is weakly continuous on $W^{d,\boldsymbol{\infty}}\left(\Omega;\boldsymbol{ \Lambda^{k} } \right)$, then for any $\phi \in C_{c}^{\infty}(\Omega),$ the integrals 
 $\int_{\Omega}\phi(x)f( \boldsymbol{ d\omega})$ and $ - \int_{\Omega}\phi(x)f( \boldsymbol{ d\omega})$ are both weakly lower semicontinuous in  
 $W^{d,\boldsymbol{\infty}}\left(\Omega;\boldsymbol{ \Lambda^{k} } \right).$ Using Theorem \ref{necessary condition semicontinuity}, we obtain that 
 $$ \boldsymbol{\xi} \mapsto \phi(x)f(\boldsymbol{\xi})$$ must be vectorially ext. quasiaffine. Since $\phi \in C_{c}^{\infty}(\Omega)$ is arbitrary, this implies 
 $\boldsymbol{\xi} \mapsto f(\boldsymbol{\xi})$ must be vectorially ext. quasiaffine. This finishes the proof.
\end{proof}
\subsection{Weak continuity of wedge products}
\subsubsection{Weak wedge products for exact forms}
Before moving on to results concerning sufficient condition for weak continuity, we first develop the notion of weak or distributional wedge products in this subsection. 
We start with some terminology for the integrability exponents. 
\begin{definition}[Admissible Sobolev and H\"{o}lder exponent]\label{Sobolevholderexponent}
Given $\boldsymbol{k},\boldsymbol{\alpha}$, we call $\boldsymbol{p}$ an \emph{admissible Sobolev exponent} (with respect to 
$\boldsymbol{\alpha}$ and $\boldsymbol{k}$), if $\boldsymbol{p} = (p_1,\ldots, p_m),$ where $1 <  p_i < \infty $ for all $1 \leq i \leq m ,$ satisfies
\begin{align}\label{sobolevexponentdefinition}
 1 + \frac{1}{n} \geq \frac{1}{\theta} = \sum_{i=1}^{m} \frac{\alpha_i}{p_i} ,  
\end{align} 
and \begin{align}\label{pbiggerthanndefinition}
     1 > \frac{1}{\theta} - \frac{1}{p_i}
    \end{align}
for all $1 \leq i \leq m .$ We call $\boldsymbol{q}$ an \emph{admissible H\"{o}lder exponent} with respect to 
$\boldsymbol{\alpha}$ and $\boldsymbol{k}$, if $\boldsymbol{q} = (q_1,\ldots, q_m)$ where $1 <  q_i \leq \infty $ for all $1 \leq i \leq m ,$ satisfies
\begin{align}\label{holderexponentdefinition}
 1  \geq \frac{1}{\rho} = \sum_{i=1}^{m} \frac{\alpha_i}{q_i} ,  
\end{align} 
and \begin{align}\label{qbiggerthanndefinition}
     1 \geq \frac{1}{\rho} - \frac{1}{q_i}
    \end{align}
for all $1 \leq i \leq m .$
\end{definition}
\begin{remark}
 Note that the assumed upper bound on $\displaystyle \frac{1}{\theta} - \frac{1}{p_i}$ is only a restriction if
 $p_{i} \geq n .$ The last inequality just means that at most one of the $q_{i}$s can be $ \infty$ and $\alpha_{i}=1$ if $q_{i} = \infty$ for some $i.$
\end{remark}

\begin{definition}[Associated exponent pair]\label{associatedexponent}
 Let $\boldsymbol{p}$ be an admissible Sobolev exponent and $\boldsymbol{q}$ be either an 
admissible Sobolev exponent or an admissible H\"{o}lder exponent with respect to given 
$\boldsymbol{\alpha}$ and $\boldsymbol{k}.$ 

\noindent We call $( \boldsymbol{p}, \boldsymbol{q})$ an \emph{associated exponent pair} if for all $i=1,\ldots, m,$ we have, 
\begin{align*}
  p_{i} &\geq \frac{nq_{i}}{n+q_{i}} &&\text{ if } q_{i} < \infty, \\
  p_{i} &\geq n &&\text{ if }q_{i}=\infty.
\end{align*}
Furthermore, if the inequalities are strict for all $1 \leq i \leq m$, we call $( \boldsymbol{p}, \boldsymbol{q})$ an \emph{associated compact exponent pair}. 
\end{definition}
\begin{remark}
 Note that if $\frac{nq_{i}}{n+q_{i}} \leq 1$ for some $i,$ then the condition $p_{i} \geq \frac{nq_{i}}{n+q_{i}}$ is not a restriction since $p_{i} > 1$ anyway.
\end{remark}\smallskip

\noindent Now we need a lemma which shows how a bound of the exterior derivative implies improved regularity of the coexact part in the  Hodge decomposition. 
\begin{lemma}\label{Hodge decomposition}
 Let $\Omega\subset\mathbb{R}^n$ be open, bounded and smooth. Let $1 \leq k \leq n.$ Let $\omega \in L^{q}( \Omega; \Lambda^{k-1}),$ $d\omega \in 
 L^{p}( \Omega; \Lambda^{k})$ with $ 1 < p < \infty$ and $1 < q \leq \infty. $ 
 Then there exists a decomposition of $\omega$ such that 
 $$ \omega = \omega_{exact} + \omega_{coexact} + \omega_{har} \qquad \text{ in } \Omega, $$
 such that 
 $\omega_{exact}$ is exact, $\omega_{har}$ is a harmonic field and $\omega_{coexact} \in W^{1,p}(\Omega; \Lambda^{k-1}).$ In other words, $\omega_{exact} = d\varphi$ with $\varphi \in W^{1,r}(\Omega; \Lambda^{k-2})$ for all $1 < r \leq q$ 
 if $q< \infty$ or $1 < r< \infty$ if $q=\infty $ and $d\omega_{har} = \delta \omega_{har} = 0$ in $\Omega.$
 Moreover, we have the estimates 
 $$ \lVert \varphi \rVert_{W^{1,r}} \leq c \lVert \omega \rVert_{L^{q}},\  \lVert \omega_{har} \rVert_{C_{loc}^{\infty}} \leq c \lVert \omega \rVert_{L^{q}} \text{ and }
\lVert \omega_{coexact} \rVert_{W^{1, p}} \leq c \lVert d\omega \rVert_{L^{p}}.$$
\end{lemma}
\begin{proof}
 Fix $1 < r < \infty$ such that $r \leq q.$ Then since $\omega \in L^{q}(\Omega; \Lambda^{k-1})$ implies $\omega \in L^{r}(\Omega; \Lambda^{k-1}),$ 
 we use Theorem 6.9(iii) of \cite{CsatoDacKneuss} 
to obtain the decomposition
\begin{align*}
\left\lbrace \begin{gathered}
              \omega = da + \delta b + h\quad \text{ and } \quad \delta a = d b= dh = \delta h = 0 \text{ in } \Omega,  \\
               \nu\wedge a = \nu\lrcorner b  = 0 \text{ on }\partial\Omega.
             \end{gathered}\right.
\end{align*}
with $a \in W_{T}^{1,r}(\Omega; \Lambda^{k-2}),$ $b \in W_{N}^{1,r}(\Omega; \Lambda^{k})$ and $h \in \mathcal{H}(\Omega; \Lambda^{k-1}). $ 
Moreover, we also have the estimates
\begin{align*}
 \lVert a \rVert_{W^{1,r}} \leq c \lVert \omega \rVert_{L^{r}} ,\qquad  \lVert h \rVert_{C_{loc}^{\infty}} \leq c \lVert \omega \rVert_{L^{r}}.
\end{align*}
Now since 
$d\omega \in L^{p}(\Omega, \Lambda^{k}),$ we see that $d(\delta b) = d\omega \in L^{p}(\Omega, \Lambda^{k}) ,$ $\delta (\delta b) = 0$ in $\Omega$ and 
$\nu\lrcorner \delta b = 0$ in $\partial\Omega,$ as $\nu\lrcorner b = 0$ in $\partial\Omega.$ Regularity result for this first order elliptic system 
implies $\delta b \in W^{1,p}$ with the estimate. Setting 
$\omega_{exact} = da$, $\omega_{har} = h$ and $\omega_{coexact} = \delta b$ concludes the proof.\end{proof}
\begin{remark}
 If we assume $\nu\wedge\omega = 0$ on $\partial\Omega,$ it is possible to use Hodge decomposition with vanishing tangential components 
 (see Theorem 6.9(i) of \cite{CsatoDacKneuss}) to prove the lemma, in which case we would also have 
 $\omega_{har} \in \mathcal{H}_{T}(\Omega; \Lambda^{k-1})$ and $\nu\wedge \omega_{coexact} = 0$ on $\partial\Omega.$
\end{remark}
\smallskip 
 
\noindent We call $\omega_{exact},\omega_{har}$ and $\omega_{coexact},$ respectively, the \emph{exact part}, \emph{harmonic part} and the \emph{coexact part} of $\omega.$ 
Now we are ready to define weak wedge products. We start with the case of exact forms first.
\begin{definition}[Weak wedge product for exact forms]\label{definitionexactweak}
Let $\Omega\subset\mathbb{R}^n$ be open, bounded and smooth. Let $\boldsymbol{p}$ be an admissible Sobolev exponent with respect to 
$\boldsymbol{\alpha}$ and $\boldsymbol{k}.$ 
Then for any componentwise exact 
$\boldsymbol{k}$-form $\boldsymbol{d\omega}=(d\omega_{1},\ldots,d\omega_{m}) \in L^{\boldsymbol{p}}(\Omega; \boldsymbol{\Lambda^{k}}) ,$ we define 
$\left(\boldsymbol{d\omega}^{\boldsymbol{\alpha}}\right)_{weak}  \in \mathcal{D}'(\Omega;\Lambda^{ \lvert \boldsymbol{k\alpha } \rvert }(\mathbb{R}^n)) ,$ by the actions
\begin{align}\label{weakexactwedgeproduct}
 &\left(\boldsymbol{d\omega}^{\boldsymbol{\alpha}}\right)_{weak}(\psi) \notag \\&:=  - (-1)^{N_{i}^{j_{i}}}\int_{\Omega} \langle \delta\psi ; 
 d\omega_{1}^{\alpha_{1}}\wedge\ldots\wedge d\omega_{i}^{j_{i}-1}\wedge \omega_{i, coexact}\wedge d\omega_{i}^{\alpha_{i} -j_{i}}\wedge \ldots \wedge 
 d\omega_{m}^{\alpha_{m}} \rangle ,  
\end{align}
for all $\psi \in 
 C_{c}^{\infty}(\Omega;\Lambda^{ \lvert \boldsymbol{k\alpha } \rvert }(\mathbb{R}^n)),$
where $\omega_{i,coexact}$ stands for the coexact part of $\omega_{i}$ and $\displaystyle N_{i}^{j_{i}} = k_{i}(j_{i} -1) + \sum_{j=1}^{i-1} k_{j}\alpha_{j}, $ 
for any $i= 1,\ldots, m,$ $j_{i} = 1, \ldots, \alpha_{i}.$ 
\end{definition}
\begin{remark}\label{weakwedgeexactremark} 
Lemma \ref{Hodge decomposition}, Sobolev embedding and the conditions \eqref{sobolevexponentdefinition} and \eqref{pbiggerthanndefinition} together ensure that the 
integrals on the right hand side of \eqref{weakexactwedgeproduct} are all finite. It is easy to see that they are also equal and if $1 \geq \frac{1}{\theta},$ then  
$$ \left(\boldsymbol{d\omega}^{\boldsymbol{\alpha}}\right)_{weak} = \boldsymbol{d\omega}^{\boldsymbol{\alpha}} \qquad \text{ in } 
\mathcal{D}'(\Omega;\Lambda^{ \lvert \boldsymbol{k\alpha } \rvert }(\mathbb{R}^n)) .$$
\end{remark}
This is not the only possible definition of weak wedge products for exact forms. We can require even less integrability on $\boldsymbol{d\omega}$ if we assume some integrability 
of $\boldsymbol{\omega}.$ The following definition is a generalization of the definition used by 
Brezis-Nguyen \cite{BrezisNguyenjacobian} for the Jacobian determinant in the classical case. 

\begin{definition}[Very weak product]\label{definitionexactveryweak}
 Let $\Omega\subset\mathbb{R}^n$ be open, bounded and smooth. Let $\boldsymbol{p},\boldsymbol{q}$ satisfy $1 < p_{i} < \infty,$ $1 < q_{i} \leq \infty$ and 
 $$ 1 \geq \frac{1}{q_{i}} + \frac{1}{\theta}- \frac{1}{p_{i}}, \qquad \text{ for all } 1 \leq i \leq m,$$
 where $ \displaystyle \frac{1}{\theta} = \sum_{i=1}^{m} \frac{\alpha_i}{p_i}.$ Then for any $\boldsymbol{\omega} \in L^{\boldsymbol{q}}(\Omega; \boldsymbol{\Lambda^{k-1}})$ with 
$\boldsymbol{d\omega} \in L^{\boldsymbol{p}}(\Omega; \boldsymbol{\Lambda^{k}}),$ we define 
$\left(\boldsymbol{d\omega}^{\boldsymbol{\alpha}}\right)_{very\ weak}  \in \mathcal{D}'(\Omega;\Lambda^{ \lvert \boldsymbol{k\alpha } \rvert }(\mathbb{R}^n)) ,$ by the actions
\begin{align}\label{veryweakexactwedgeproduct}
 &\left(\boldsymbol{d\omega}^{\boldsymbol{\alpha}}\right)_{very\ weak}(\psi) \notag \\&:=  - (-1)^{N_{i}^{j_{i}}}\int_{\Omega} \langle \delta\psi ; 
 d\omega_{1}^{\alpha_{1}}\wedge\ldots\wedge d\omega_{i}^{j_{i}-1}\wedge \omega_{i}\wedge d\omega_{i}^{\alpha_{i} -j_{i}}\wedge \ldots \wedge d\omega_{m}^{\alpha_{m}}\rangle ,  
\end{align}
for all $\psi \in 
 C_{c}^{\infty}(\Omega;\Lambda^{ \lvert \boldsymbol{k\alpha } \rvert }(\mathbb{R}^n)),$
where $\displaystyle N_{i}^{j_{i}} = k_{i}(j_{i} -1) + \sum_{j=1}^{i-1} k_{j}\alpha_{j}, $ 
for any  $i= 1,\ldots, m,$ $j_{i} = 1, \ldots, \alpha_{i}.$ 
\end{definition}
Note that there are integrability exponents for which only one of them is well-defined. Even in the classical case, for the Jacobian determinant of a function 
$u \in W^{1, \frac{n^2}{n+1}}(\Omega;\mathbb{R}^{n}),$ only the first one is defined and for a function 
$u \in W^{1, n-1}(\Omega;\mathbb{R}^{n})\cap L^{\infty}(\Omega;\mathbb{R}^{n}),$ only the second one is defined. However, it is not difficult to show that when both are well-defined, we have, 
$$\left(\boldsymbol{d\omega}^{\boldsymbol{\alpha}}\right)_{weak} = \left(\boldsymbol{d\omega}^{\boldsymbol{\alpha}}\right)_{very\ weak} \qquad \text{ in } 
\mathcal{D}'(\Omega;\Lambda^{ \lvert \boldsymbol{k\alpha } \rvert }(\mathbb{R}^n)) . $$

We also have the following general telescopic estimate. 
\begin{lemma}\label{telescopicexactwithoutinfty}
Let $\Omega\subset\mathbb{R}^n$ be open, bounded and smooth. Let $\boldsymbol{k},\boldsymbol{\alpha},\boldsymbol{p}, \boldsymbol{q}$ be given. Let $\boldsymbol{\mu}$ be given by, 
$$1 = \frac{1}{\mu_{i}} + \frac{1}{\theta} - \frac{1}{p_{i}} \qquad \text{ for all }1 \leq i \leq m.$$ 
\begin{itemize}
 \item[(i)] If $\boldsymbol{p}$ is an admissible Sobolev exponent, then for any two componentwise exact 
$\boldsymbol{k}$-form $\boldsymbol{d\xi}, \boldsymbol{d\zeta} \in L^{\boldsymbol{p}}(\Omega; \boldsymbol{\Lambda^{k}}),$ there exists a constant $C > 0$ such that 
  \begin{align*}
            \Big\lvert \big[ &\left( \boldsymbol{d\xi}^{\boldsymbol{\alpha}} \right)_{weak}  
            - \left(\boldsymbol{d\zeta}^{\boldsymbol{\alpha}}\right)_{weak} \big](\psi) \Big\rvert \\ 
            &\begin{aligned}\leq C  \sum_{i=1}^{m} \alpha_{i}\left\lVert \delta \psi \right\rVert_{\infty}\left\lVert \xi_{i,coexact} - \zeta_{i,coexact} \right\rVert_{\mu_{i}} 
            \left( \left\lVert d\xi_{i} \right\rVert_{p_{i}} 
            + \left\lVert d\zeta_{i} \right\rVert_{p_{i}} \right)^{\alpha_{i}-1}  \\ \prod\limits_{\substack{j=1\\ j \neq i}}^{m} \left( 
            \left\lVert d\xi_{j} \right\rVert_{p_{j}} 
            + \left\lVert d\zeta_{j} \right\rVert_{p_{j}} \right)^{\alpha_{j}},\end{aligned}\end{align*}
           for all $\psi \in C_{c}^{\infty}(\Omega;\Lambda^{ \lvert \boldsymbol{k\alpha } \rvert }(\mathbb{R}^n)).$ 
 \item[(ii)] If $\boldsymbol{p},\boldsymbol{q}$ are as in definition \ref{definitionexactveryweak}, then for any $\boldsymbol{\xi},\boldsymbol{\zeta} \in 
 L^{\boldsymbol{q}}(\Omega; \boldsymbol{\Lambda^{k-1}})$ with $\boldsymbol{d\xi}, \boldsymbol{d\zeta} \in L^{\boldsymbol{p}}(\Omega; \boldsymbol{\Lambda^{k}}),$ there exists a constant $C > 0$ such that 
  \begin{align*}
            \Big\lvert \big[ \left( \boldsymbol{d\xi}^{\boldsymbol{\alpha}} \right)_{very\ weak}  
            &- \left(\boldsymbol{d\zeta}^{\boldsymbol{\alpha}}\right)_{very\ weak} \big](\psi) \Big\rvert \\ 
            &\begin{aligned}\leq C  \sum_{i=1}^{m} \alpha_{i}\left\lVert \delta \psi \right\rVert_{\infty}\left\lVert \xi_{i} - \zeta_{i} \right\rVert_{\mu_{i}} 
            \left( \left\lVert d\xi_{i} \right\rVert_{p_{i}} 
            + \left\lVert d\zeta_{i} \right\rVert_{p_{i}} \right)^{\alpha_{i}-1}  \\ \prod\limits_{\substack{j=1\\ j \neq i}}^{m} \left( 
            \left\lVert d\xi_{j} \right\rVert_{p_{j}} 
            + \left\lVert d\zeta_{j} \right\rVert_{p_{j}} \right)^{\alpha_{j}},\end{aligned}\end{align*}
           for all $\psi \in C_{c}^{\infty}(\Omega;\Lambda^{ \lvert \boldsymbol{k\alpha } \rvert }(\mathbb{R}^n)).$          
\end{itemize}
\smallskip
\end{lemma}
\begin{proof}
 It is just a matter of rewriting as a telescopic sum. We show only one, the other being similar. Note that we have, 
 \begin{align*}
  &\big[ \left( \boldsymbol{d\xi}^{\boldsymbol{\alpha}} \right)_{weak}  
            - \left(\boldsymbol{d\zeta}^{\boldsymbol{\alpha}}\right)_{weak} \big](\psi) 
            \\ &= \sum_{i=1}^{m}\sum_{j=1}^{\alpha_{i}} \left( d\zeta_{1}^{\alpha_{1}}\wedge\ldots\wedge d\zeta_{i}^{j-1}\wedge d(\xi_{i}-\zeta_{i})\wedge d\xi^{\alpha_{i} -j}_{i}\wedge\ldots
            \wedge d\xi_{m}^{\alpha_{m}}\right)_{weak} (\psi) .\end{align*}
            Using the definition of weak wedge product, the estimate follows from H\"{o}lder inequality.\end{proof}\smallskip
            
\noindent This immediately implies the weak continuity results for wedge product of exact forms.  
\begin{theorem}\label{weakcontexact}
 Let $\Omega\subset\mathbb{R}^n$ be open, bounded and smooth. Let $\boldsymbol{k},\boldsymbol{\alpha}$ be given.
 \begin{itemize}
  \item[(i)] Let $\boldsymbol{p}$ be an admissible Sobolev exponent such that $\displaystyle 1 + \frac{1}{n} > \frac{1}{\theta},$ and 
  $\boldsymbol{d\xi_{s}} \boldsymbol{\rightharpoonup} \boldsymbol{d\xi}$ in $L^{\boldsymbol{p}}(\Omega; \boldsymbol{\Lambda^{k}}),$ then 
  $$ \left( \boldsymbol{d\xi_{s}^{\alpha}} \right)_{weak} \boldsymbol{\rightharpoonup} \left( \boldsymbol{d\xi^{\alpha}}\right)_{weak} \quad \text{ in }
\mathcal{D}'(\Omega;\Lambda^{ \lvert \boldsymbol{k\alpha } \rvert }(\mathbb{R}^n)). $$
Moreover, if $\displaystyle 1 \geq \frac{1}{\theta},$ then 
$$ \boldsymbol{d\xi_{s}^{\alpha}}  \boldsymbol{\rightharpoonup}  \boldsymbol{d\xi^{\alpha}} \quad \text{ in }
\mathcal{D}'(\Omega;\Lambda^{ \lvert \boldsymbol{k\alpha } \rvert }(\mathbb{R}^n)). $$
If $\displaystyle 1 > \frac{1}{\theta},$ then we also have,
$$ \boldsymbol{d\xi_{s}^{\alpha}}  \boldsymbol{\rightharpoonup}  \boldsymbol{d\xi^{\alpha}} \quad \text{ in }
L^{\theta}(\Omega;\Lambda^{ \lvert \boldsymbol{k\alpha } \rvert }(\mathbb{R}^n)). $$

\item[(ii)] Let $\boldsymbol{p},\boldsymbol{q}$ be as in definition \ref{definitionexactveryweak} and 
$\boldsymbol{d\xi_{s}} \boldsymbol{\rightharpoonup} \boldsymbol{d\xi}$ in $L^{\boldsymbol{p}}(\Omega; \boldsymbol{\Lambda^{k}})$ and 
$\boldsymbol{\xi_{s}} \boldsymbol{\rightarrow} \boldsymbol{\xi}$ in $L^{\boldsymbol{q}}(\Omega; \boldsymbol{\Lambda^{k-1}}),$ then 
$$ \left(\boldsymbol{d\xi_{s}^{\alpha}}\right)_{very\ weak}  \boldsymbol{\rightharpoonup}  \left(\boldsymbol{d\xi^{\alpha}}\right)_{very\ weak} \quad \text{ in }
\mathcal{D}'(\Omega;\Lambda^{ \lvert \boldsymbol{k\alpha } \rvert }(\mathbb{R}^n)). $$
 \end{itemize}
\end{theorem}
\begin{proof}
 The second conclusion is immediate form the telescopic estimate. For the first one, note that the hypotheses on $\boldsymbol{p}$ implies that the embeddings 
 $W^{1,p_{i}} \hookrightarrow L^{\mu_{i}}$ are compact for all $1 \leq i \leq m.$ Thus $d\omega_{s,i} \rightharpoonup d\omega_{i}$ in $L^{p_{i}}$ implies 
 $$\lVert \omega_{s,i,coexact} - \omega_{i,coexact}\rVert_{\mu_{i}} \rightarrow 0$$ for all $1 \leq i \leq m.$ The convergence in distribution follows. The weak convergence in 
 $L^{\theta}$ follows from the fact that in that case, $\lbrace \boldsymbol{d\xi_{s}^{\alpha}}  \rbrace$ is uniformly bounded in $L^{\theta}$ and thus has a weak limit in 
 $L^{\theta}.$ Uniqueness of the weak limit concludes the proof. 
\end{proof}

\subsubsection{Weak wedge product for general forms}
The first definition, i.e the definition of weak wedge products for exact forms can be used, together with Hodge decomposition to define weak wedge products for general forms 
$\boldsymbol{\omega}$ with 
some integrability of $\boldsymbol{d\omega}.$ To fix ideas, we start with two forms  $v_{1} \in W^{d, p_{1}}\left(\Omega;\Lambda^{k_{1}}(\mathbb{R}^{n})\right),$ 
$v_{2} \in W^{d, p_{2}}\left(\Omega;\Lambda^{k_{2}}(\mathbb{R}^{n})\right),$ with $\displaystyle  1 + \frac{1}{n} \geq \frac{1}{p_{1}} + 
\frac{1}{p_{2}} ,$ $1 < p_{1},p_{2} < \infty.$ Using Hodge decomposition, we have, formally,
\begin{align}\label{hodgeidentity}
  v_{1} \wedge v_{2} &= (da_{1} + \delta b_{1} + h_{1})\wedge (da_{2} + \delta b_{2} + h_{2}) \notag \\
  &= da_{1}\wedge da_{2} + da_{1}\wedge( \delta b_{2} + h_{2}) + ( \delta b_{1} + h_{1})\wedge ( da_{2} + \delta b_{2} + h_{2}).
\end{align}
 Note that by lemma \ref{Hodge decomposition}, Sobolev embedding and H\"{o}lder inequality, every term except the first in the right hand side of \eqref{hodgeidentity} is indeed in $L^{1}.$
But the first term $da_{1}\wedge da_{2}$ is a wedge product of exact forms and we can use the notion of weak wedge product in such cases. Using that definition, we can now define
 $$ \left( v_{1} \wedge v_{2} \right)_{weak} := \left( da_{1}\wedge da_{2}\right)_{weak} + da_{1}
 \wedge( \delta b_{2} + h_{2}) + ( \delta b_{1} + h_{1})\wedge ( da_{2} + \delta b_{2} + h_{2}). $$

\noindent  Observe also that the regularity of $da_{i}$ depends on the regularity of $v_{i},$ whereas the improved regularity of 
$\delta b_{i} + h_{i}$ comes from the regularity of $dv_{i}.$ 
Suppose $v_{1} \in L^{q_{1}}\left(\Omega;\Lambda^{k_{1}}(\mathbb{R}^{n})\right)$ with  $dv_{1} \in  L^{p_{1}}\left(\Omega;\Lambda^{k_{1}+1}(\mathbb{R}^{n})\right)$
and $v_{2} \in L^{q_{2}}\left(\Omega;\Lambda^{k_{2}}(\mathbb{R}^{n})\right)$ with $dv_{2} \in  L^{p_{2}}\left(\Omega;\Lambda^{k_{2}+1}(\mathbb{R}^{n})\right),$ where 
$1 < q_{1}, q_{2}, p_{1}, p_{2} < \infty,$ $\displaystyle \frac{1}{q_{1}} + \frac{1}{q_{2}} \leq 1 ,$ $\displaystyle \frac{1}{p_{1}} + \frac{1}{p_{2}} \leq 1 + \frac{1}{n}, $ and 
$p_{i} \geq \frac{nq_{1}}{n+ q_{i}} $ for $i=1,2.$ Then we have 
$da_{1}\wedge da_{2} \in L^{1}$ and we obtain $$ \left( da_{1}\wedge da_{2}\right)_{weak} = da_{1}\wedge da_{2} \qquad \text{ in }
\mathcal{D}'(\Omega;\Lambda^{ k_{1} + k_{2} }(\mathbb{R}^n)) .$$
But since $\displaystyle \frac{1}{p_{1}} + \frac{1}{p_{2}} \leq 1 + \frac{1}{n}, $ all other terms are in $L^{1}$ as before. Thus, we obtain,
$$  v_{1} \wedge v_{2} = \left( v_{1} \wedge v_{2} \right)_{weak} \quad \text{ in }
\mathcal{D}'(\Omega;\Lambda^{ k_{1} + k_{2} }(\mathbb{R}^n)) .$$ 

All of these can be done for the general case. If $\boldsymbol{p}$ is an admissible Sobolev exponent, 
then given $\boldsymbol{\omega} \in W^{d,\boldsymbol{p}}(\Omega; \boldsymbol{\Lambda^{k}}),$ we can define the distribution 
\begin{align*}
 \left( \boldsymbol{\omega}^{\boldsymbol{\alpha}} \right)_{weak} = \left( \left(\boldsymbol{\omega}_{exact}\right)^{\boldsymbol{\alpha}} \right)_{weak} 
 &+\text{ all other terms in the } \\
 &\begin{aligned}[t]
   \text{ formal expansion of }
 \left( \boldsymbol{\omega}_{exact} \right.&\left. + \boldsymbol{\omega}_{coexact}  + \boldsymbol{\omega}_{har} \right)^{\boldsymbol{\alpha}} \\ &\text{ in } 
 \mathcal{D}'(\Omega;\Lambda^{ \lvert \boldsymbol{k\alpha } \rvert }(\mathbb{R}^n)).
  \end{aligned}                                                                                                                                               
 \end{align*}
Using this definition, we can prove the following result, due to Iwaniec \cite{Iwaniecnonlinearcommutator}, which is 
a generalization of the classical `div-curl' lemma or `compensated compactness' lemma of Murat \cite{Muratsufficientnecessarycompensated} and Tartar \cite{TaCC}.
\begin{theorem}\label{weakcontgeneral}
 Let $\Omega\subset\mathbb{R}^n$ be open, bounded and smooth. Let $\boldsymbol{k},\boldsymbol{\alpha}$ be given. 
 Let $\boldsymbol{p}$ be an admissible Sobolev exponent such that $\displaystyle 1 + \frac{1}{n} > \frac{1}{\theta}.$
\begin{itemize}
 \item[(i)]  Let $\boldsymbol{\xi_{s}} \boldsymbol{\rightharpoonup} \boldsymbol{\xi} $ in $W^{d,\boldsymbol{p}}(\Omega; \boldsymbol{\Lambda^{k}}).$ 
Then 
$$ \left( \boldsymbol{\xi_{s}^{\alpha}} \right)_{weak} \boldsymbol{\rightharpoonup} \left( \boldsymbol{\xi^{\alpha}}\right)_{weak} \quad \text{ in }
\mathcal{D}'(\Omega;\Lambda^{ \lvert \boldsymbol{k\alpha } \rvert }(\mathbb{R}^n)). $$
Moreover, if $\displaystyle 1 \geq \frac{1}{\theta},$ then 
$$ \boldsymbol{\xi_{s}^{\alpha}}  \boldsymbol{\rightharpoonup}  \boldsymbol{\xi^{\alpha}} \quad \text{ in }
\mathcal{D}'(\Omega;\Lambda^{ \lvert \boldsymbol{k\alpha } \rvert }(\mathbb{R}^n)). $$
If $\displaystyle 1 > \frac{1}{\theta},$ then we also have,
$$ \boldsymbol{\xi_{s}^{\alpha}}  \boldsymbol{\rightharpoonup}  \boldsymbol{\xi^{\alpha}} \quad \text{ in }
L^{\theta}(\Omega;\Lambda^{ \lvert \boldsymbol{k\alpha } \rvert }(\mathbb{R}^n)). $$
\item[(ii)]Let $\boldsymbol{q}$ be an admissible H\"{o}lder exponent such that $(\boldsymbol{p},\boldsymbol{q})$ is an associated  
compact exponent pair. Let $\boldsymbol{\xi_{s}} \boldsymbol{\rightharpoonup} \boldsymbol{\xi} $ in $L^{\boldsymbol{q}}(\Omega; \boldsymbol{\Lambda^{k}})$ and  
$\boldsymbol{d\xi_{s}} \boldsymbol{\rightharpoonup} \boldsymbol{d\xi} $ in $L^{\boldsymbol{p}}(\Omega; \boldsymbol{\Lambda^{k+1}}).$ Then 
$$ \boldsymbol{\xi_{s}^{\alpha}}  \boldsymbol{\rightharpoonup}  \boldsymbol{\xi^{\alpha}} \quad \text{ in }
\mathcal{D}'(\Omega;\Lambda^{ \lvert \boldsymbol{k\alpha } \rvert }(\mathbb{R}^n)). $$
If $\displaystyle 1 > \frac{1}{\rho},$ then we also have,
$$ \boldsymbol{\xi_{s}^{\alpha}}  \boldsymbol{\rightharpoonup}  \boldsymbol{\xi^{\alpha}} \quad \text{ in }
L^{\rho}(\Omega;\Lambda^{ \lvert \boldsymbol{k\alpha } \rvert }(\mathbb{R}^n)). $$
\end{itemize}
\end{theorem}

\section{Existence of minimizers}\label{existencetheorems}
In this section, we discuss existence theorems for minimization problems. But first we begin by showing that unlike the classical calculus of variations, here in general 
we can not always expect a minimizer to exist if the integrand depends explicitly on $\boldsymbol{\omega}.$

\subsection{Nonexistence results}
Even when 
the explicit dependence on $\omega$ is a convex, additive term, we have the following counterexample already for $m=1$, as soon as $k \geq 2.$ 
\begin{proposition}[Counterexample to existence of minimizer]\label{failure of existence}
Let  $n \geq 2.$ Also let $2 \leq k \leq n$ and let $\Omega \subset\mathbb{R}^n$ be open, bounded and smooth and contractible. 
Then for any $\omega_{0} \in W^{1,2}(\Omega;\Lambda^{k-1})$ with $\nu\wedge\omega_{0} = 0$  but $\omega_{0} \neq 0$ on $\partial \Omega ,$ the problem  
\[
 \inf\left\{  I(\omega) =\frac{1}{2}\int_{\Omega}\lvert d\omega \rvert^{2} +\frac{1}{2}\int_{\Omega} \lvert \omega \rvert^{2}  :\omega \in 
 \omega_{0}+W_{0}^{1,2}\left(  \Omega;\Lambda^{k-1}\right)  \right\}  =m,
\]
does not admit a minimizer.
\end{proposition}
\begin{proof} Suppose the problem admits a minimizer $\alpha \in \omega_{0}+W_{0}^{1,2}\left(  \Omega;\Lambda^{k-1}\right).$ Then $\alpha$
satisfies the weak form of the Euler-Lagrange equation, i.e 
\begin{equation*}
   \int_{\Omega} \langle d\alpha, d\phi \rangle +  \int_{\Omega} \langle \alpha, \phi \rangle = 0 \qquad \text{ for all } \phi \in W_{0}^{1,2}\left(\Omega; \Lambda^{k-1}\right).
\end{equation*} 
Choosing $\phi = d\theta$ for some $\theta \in C_{c}^{\infty}\left(  \Omega;\Lambda^{k-2}\right),$ we see immediately that this implies $\delta\alpha = 0$ in distributions.
Now for any $\psi \in W_{T}^{d,2}\left(\Omega; \Lambda^{k-1}\right),$ there exist $\phi \in W_{0}^{1,2}\left(\Omega; \Lambda^{k-1}\right)$ and 
$\eta \in W_{0}^{1,2}\left(\Omega; \Lambda^{k-2}\right)$ such that 
$$ \psi = \phi + d\eta .$$ Indeed, since $\Omega$ is contractible, we can solve the following two problems one after another (see e.g Theorem 8.16 in \cite{CsatoDacKneuss}).
\begin{align*}
\left\lbrace \begin{aligned}
                d \phi &= d\psi &&\text{ in } \Omega, \\
                \phi &= 0 &&\text{  on } \partial\Omega.
               \end{aligned} \right.  \quad \text{ and } \quad \left\lbrace \begin{aligned}
                d \eta &= \psi -\phi &&\text{ in } \Omega, \\
                \eta &= 0 &&\text{  on } \partial\Omega.
               \end{aligned} \right.\end{align*}
This gives the desired decomposition. Thus,  we have, 
\begin{equation*}
   \int_{\Omega} \langle d\alpha, d\psi \rangle +  \int_{\Omega} \langle \alpha, \psi \rangle = 
   0 \qquad \text{ for all } \psi \in W_{T}^{d,2}\left(\Omega; \Lambda^{k-1}\right).
\end{equation*} 
But this implies $\alpha$ is also a minimizer of the problem 
\[
 \inf\left\{  I(\omega) =\frac{1}{2}\int_{\Omega}\lvert d\omega \rvert^{2} +\frac{1}{2}\int_{\Omega} \lvert \omega \rvert^{2}  :\omega \in 
 W_{\delta, T}^{d,2}\left(  \Omega;\Lambda^{k-1}\right)  \right\}  =m.
\]
But it is easy to show that the minimizer of this problem is unique and $0$ is a minimizer. Thus $\alpha = 0,$ which is impossible since $\omega_{0} \neq 0$ on $\partial\Omega.$ 
This concludes the proof.
\end{proof}
\begin{remark} This counterexample can easily be generalized for any $1 < p < \infty.$ Also note that the term depending on $d\omega$ is convex, thus ext. polyconvex and ext. 
quasiconvex as well.  
\end{remark}
\subsection{Existence theorems}
In view of the previous subsection, we can expect general existence theorems to hold only when the explicit dependence on $\boldsymbol{\omega}$ is rather special, if any. We now 
show that an additive term which is linear in $\boldsymbol{\omega},$ still allows fairly general existence results. We start with a lemma. 
\begin{lemma}\label{lemmaforhodge}
 Let $\boldsymbol{p} = (p_1,\ldots, p_m)$ where $1 < p_i < \infty$ for all $1 \leq i \leq m .$ Let $\boldsymbol{\omega_{0}} \in 
W^{1,\boldsymbol{p}}\left(\Omega;\boldsymbol{\Lambda^{k -1}}\right)$ be given. Let $\lbrace \boldsymbol{\omega}^{s} \rbrace \subset \boldsymbol{\omega_{0}} + 
W^{d,\boldsymbol{p}}_{T}\left(\Omega;\boldsymbol{\Lambda^{k -1}}\right)$ be a sequence such that 
$\lVert \boldsymbol{d\omega}^{s}\rVert_{L^{\boldsymbol{p}}\left(\Omega;\boldsymbol{\Lambda^{k}}\right)}$ is uniformly bounded. Then there exist 
$\boldsymbol{\omega} \in \boldsymbol{\omega_{0}} + W^{1,\boldsymbol{p}}_{0}\left(\Omega;\boldsymbol{\Lambda^{k -1}}\right),$
$\boldsymbol{\beta} \in \boldsymbol{\omega_{0}} + W^{1,\boldsymbol{p}}_{\boldsymbol{\delta}, T}\left(\Omega;\boldsymbol{\Lambda^{k -1}}\right)$ satisfying 
$$\boldsymbol{d\beta} =  \boldsymbol{d\omega} \quad \text{ in } \Omega, $$ and a sequence 
$\lbrace \boldsymbol{\beta}^{s} \rbrace \subset \boldsymbol{\omega_{0}} + 
W^{1,\boldsymbol{p}}_{\boldsymbol{\delta}, T}\left(\Omega;\boldsymbol{\Lambda^{k -1}}\right)$ such that 
\begin{gather*}
\boldsymbol{d\beta}^{s} =  \boldsymbol{d\omega}^{s} \quad \text{ in } \Omega,\text{ for every } s \\
\text{ and } \\ 
 \boldsymbol{\beta}^{s} \boldsymbol{\rightharpoonup} \boldsymbol{\beta} \text{ in } W^{d,\boldsymbol{p}}\left(\Omega;\boldsymbol{\Lambda^{k -1}}\right). 
\end{gather*}
\end{lemma}
\begin{proof}
 First for every $s$, we find $\boldsymbol{\beta}^{s} \in \boldsymbol{\omega_{0}} + W^{1, \boldsymbol{p}}_{\boldsymbol{\delta}, T}(\Omega; \boldsymbol{\Lambda^{k}}),$
such that,
\begin{equation*}
   \left\lbrace \begin{aligned}
                \boldsymbol{d\beta}^{s} = \boldsymbol{d\omega}^{s}  \quad &\text{and} \quad  \boldsymbol{\delta \beta}^{s} = 0 &&\text{ in } \Omega, \\
                \nu\wedge \boldsymbol{\beta}^{s} = \nu\wedge &\boldsymbol{\omega}^{s} = \nu\wedge \boldsymbol{\omega_{0}} &&\text{  on } \partial\Omega,
                \end{aligned} 
                \right. 
                \end{equation*}
and there exist constants $c_{1},c_{2} > 0$ such that
$$ \lVert \boldsymbol{\beta}^{s} \rVert_{W^{1,\boldsymbol{p}}} \leq c_{1}\left\lbrace \lVert \boldsymbol{d\omega}^{s} \rVert_{L^{\boldsymbol{p}}} + 
\lVert \boldsymbol{\omega_{0}} \rVert_{W^{1,\boldsymbol{p}}}\right\rbrace \leq c_{2}.$$
Therefore, up to the extraction of a subsequence which we do not relabel, there exists 
$\boldsymbol{\beta} \in \boldsymbol{\omega_{0}} + W^{1,\boldsymbol{p}}_{\boldsymbol{\delta}, T}\left(\Omega;\boldsymbol{\Lambda^{k -1}}\right)$ such that 
$$ \boldsymbol{\beta}^{s} \boldsymbol{\rightharpoonup} \boldsymbol{\beta}\qquad \text{ in } W^{1,\boldsymbol{p}}\left(\Omega;\boldsymbol{\Lambda^{k -1}}\right).$$
Since $\nu\wedge \boldsymbol{\beta} = \nu\wedge \boldsymbol{\omega_{0}}$  on $\partial\Omega$, we can find $\boldsymbol{\omega} \in \boldsymbol{\omega_{0}} + W^{1,\boldsymbol{p}}_{0}\left(\Omega;\boldsymbol{\Lambda^{k -1}}\right)$ such that 
\begin{equation*}
   \left\lbrace \begin{aligned}
                \boldsymbol{d\omega} = \boldsymbol{d \beta}  \qquad &\text{ in } \Omega, \\
                \boldsymbol{\omega} = \boldsymbol{\alpha_{0}} \qquad &\text{  on } \partial\Omega.
                \end{aligned} 
                \right. 
                \end{equation*}
This concludes the proof.
\end{proof}

\subsubsection{Existence theorem for quasiconvex functions}
\begin{theorem}\label{Thm existence de minima ext vect}
 Let $\boldsymbol{p} = (p_1,\ldots, p_m)$ where $1 < p_i < \infty$ for all $1 \leq i \leq m .$
 Let $\Omega\subset\mathbb{R}^n$ be open, bounded, smooth. 
 Let $f:\Omega\times\boldsymbol{\Lambda^k}\rightarrow \mathbb{R}$ be a Carath\'eodory function, satisfying for a.e $x \in \Omega,$ 
 for every $\boldsymbol{\xi}= (\xi_{1}, \ldots, \xi_{m}) \in \boldsymbol{\Lambda^{k}},$
 \begin{equation}\label{growthforquasiexistence}
  \begin{gathered}[b]
  \boldsymbol{\xi} \mapsto f(x, \boldsymbol{\xi}) \text{ is vectorially ext. quasiconvex}, \\
  \gamma_{1}(x) + \sum_{i=1}^{m}\alpha_{1,i}\lvert \xi_{i} \rvert^{p_{i}} \leq f(x, \boldsymbol{\xi}) \leq \gamma_{2}(x) + \sum_{i=1}^{m}\alpha_{2,i}\lvert \xi_{i} \rvert^{p_{i}},
 \end{gathered}
 \end{equation}
where $\alpha_{2,i} \geq \alpha_{1,i} > 0$ for all $1 \leq i \leq m$ and $\gamma_{1}, \gamma_{2} \in L^{1}(\Omega).$ Let $\boldsymbol{g}
\in L^{\boldsymbol{p^{'}}}(\Omega; \boldsymbol{\Lambda^{k-1}})$ be such that
$\boldsymbol{\delta g} = 0$ in the sense of distributions and $\boldsymbol{\omega_{0}} \in W^{1,\boldsymbol{p}}\left(  \Omega;\boldsymbol{\Lambda^{k-1}}\right).$
Let \[
\mathcal{(P}_{0}\mathcal{)\quad}\inf\left\{   I(\boldsymbol{\omega}) =\int_{\Omega}\left[ f\left( x, 
\boldsymbol{d\omega}\right) + \langle \boldsymbol{g} ; \boldsymbol{\omega} \rangle \right]  :\boldsymbol{\omega} \in \boldsymbol{\omega_{0}}+W_{0}^{1,\boldsymbol{p}}
\left(  \Omega;\boldsymbol{\Lambda
^{k-1}}\right)  \right\}  =m.
\]
Then the problem $\mathcal{(\mathcal{P}}_{0}\mathcal{)}$ has a minimizer.
\end{theorem}
\begin{remark}\label{remarkforlinearlowerordervect}
\begin{enumerate}
 \item[(i)] If $k_{i}=1$ for some $i \in \lbrace 1,\ldots, m \rbrace , $ the condition $\delta g_{i} = 0$ in the sense of distributions, is automatically satisfied for all $g_{i} \in L^{p_{i}'}(\Omega)$ and hence is not a restriction.
 \item[(ii)] However, as soon as $k_{i} \geq 2 $ for some $i \in \lbrace 1,\ldots, m \rbrace , $ $g_{i}$ being coclosed is a non-trivial restriction and the 
 theorem does not hold without this assumption. In fact, we can show 
 that if $\mathcal{(\mathcal{P}}_{0} \mathcal{)}$ admits a minimizer and $2 \leq k_{i} \leq n$ for some $i \in \lbrace 1,\ldots, m \rbrace , $ then we must have 
 $\delta g_{i} = 0 $ in the sense of distributions. Indeed, suppose $ \boldsymbol{\omega}\in\boldsymbol{\omega_{0}} +W_{0}^{1,\boldsymbol{p}}\left(  \Omega;\boldsymbol{\Lambda
^{k-1}}\right) $ is a minimizer for $\mathcal{(\mathcal{P}}_{0}\mathcal{)}.$ Now if $\delta g_{i} \neq 0 ,$ since $k_{i} \geq 2,$ there exists a $\theta \in C_{c}^{\infty}(\Omega; \Lambda^{k-2})$ such that
$\displaystyle \int_{\Omega} \langle g_{i} ; d \theta \rangle < 0.$
Define $\boldsymbol{\theta}= (\theta_{1}, \ldots, \theta_{m})$ such that for all $1 \leq j \leq m,$
\begin{align*}
 \theta_{j} = \left\lbrace \begin{aligned}
                            &\theta \quad \text{ if } i=j,\\
                            &0 \quad \text{ otherwise}.
                           \end{aligned}\right.
\end{align*}
Then $\boldsymbol{\omega} + \boldsymbol{d\theta}\in\boldsymbol{\omega_{0}} +W_{0}^{1,\boldsymbol{p}}\left(  \Omega;\boldsymbol{\Lambda
^{k-1}}\right) $ and we have,
\begin{align*}
  I(\boldsymbol{\omega} + \boldsymbol{d\theta)})= \int_{\Omega}\left[ f\left( x, 
\boldsymbol{d\omega}\right) + \langle \boldsymbol{g} ; \boldsymbol{\omega} \rangle \right] + \int_{\Omega} \langle g_{i} ; d \theta \rangle  < m,
\end{align*}
which is impossible since $\boldsymbol{\omega}$ is a minimizer. 
\item[(iii)] Note that if $f:\Omega\times\boldsymbol{\Lambda^{k}}\rightarrow\mathbb{R}$ satisfies the hypotheses of the theorem for some $\boldsymbol{p},$ 
then for any $\boldsymbol{G} \in L^{\boldsymbol{p'}}\left( \Omega ; \boldsymbol{\Lambda^{k}}\right),$ 
the function $F:\Omega\times\boldsymbol{\Lambda^{k}}\rightarrow\mathbb{R}$, defined by,
$$ F(x, \boldsymbol{\xi}) = f(x,\boldsymbol{\xi}) + \langle \boldsymbol{G} ; \boldsymbol{\xi} \rangle  \quad \text{ for every } \boldsymbol{\xi} \in \boldsymbol{\Lambda^{k}}, $$
also satisfies all the hypotheses with the same $\boldsymbol{p}$. 
\end{enumerate}
\end{remark}
\begin{proof}
 \emph{Step 1} First we show that we can assume $\boldsymbol{g}=0.$ Since $\boldsymbol{g}
\in L^{\boldsymbol{p^{'}}}(\Omega; \boldsymbol{\Lambda^{k-1}})$ satisfies
$\boldsymbol{\delta g} = 0$ in the sense of distributions, we can find $\boldsymbol{G} \in W^{1, \boldsymbol{p^{'}}}(\Omega; \boldsymbol{\Lambda^{k}}),$
such that,
\begin{equation*}
   \left\lbrace \begin{aligned}
                \boldsymbol{dG} = 0  \quad &\text{and} \quad  \boldsymbol{\delta G} = \boldsymbol{g} &&\text{ in } \Omega, \\
                \nu\wedge \boldsymbol{G} &= 0 &&\text{  on } \partial\Omega.
                \end{aligned} 
                \right. 
                \end{equation*}
Thus, for any $\boldsymbol{\omega} \in \boldsymbol{\omega_{0}}+W_{0}^{1,\boldsymbol{p}}
\left(  \Omega;\boldsymbol{\Lambda^{k-1}}\right),$ we have,
\begin{align*}
 \int_{\Omega} \langle \boldsymbol{g} ; \boldsymbol{\omega} \rangle = \int_{\Omega} \langle \boldsymbol{\delta G} ; \boldsymbol{\omega} \rangle  
= - \int_{\Omega} \langle  \boldsymbol{G} ; \boldsymbol{d\omega} \rangle
+ \int_{\partial\Omega} \langle \nu\lrcorner \boldsymbol{G} ; \boldsymbol{\omega_{0}} \rangle .
\end{align*}
Given $\boldsymbol{\omega_{0}}\in
W^{1,\boldsymbol{p}}\left(  \Omega;\boldsymbol{\Lambda^{k-1}}\right) $ and $\boldsymbol{g} \in L^{\boldsymbol{p^{'}}}(\Omega; \boldsymbol{\Lambda^{k-1}})$, 
$\int_{\partial\Omega} \langle \nu\lrcorner \boldsymbol{G} ; \boldsymbol{\omega_{0}} \rangle$ is just a real number which does not matter for minimization.
Now the claim follows from remark \ref{remarkforlinearlowerordervect}(iii).\smallskip

\emph{Step 2} By step $1$, we assume from now on that $\boldsymbol{g}=0.$ Let $\lbrace \boldsymbol{\omega}^{s} \rbrace$ be a minimizing sequence of 
$\mathcal{(\mathcal{P}}_{0}\mathcal{)}$. By the growth condition \eqref{growthforquasiexistence}, there exists a constant $c>0$ such that 
$$\lVert \boldsymbol{d\omega}^{s}\rVert_{L^{\boldsymbol{p}}\left(\Omega;\boldsymbol{\Lambda^{k}}\right)} \leq c .$$ Hence by lemma \ref{lemmaforhodge}, there exist maps 
$\boldsymbol{\omega} \in \boldsymbol{\omega_{0}} + W^{1,\boldsymbol{p}}_{0}\left(\Omega;\boldsymbol{\Lambda^{k -1}}\right)$ and 
$\boldsymbol{\beta} \in \boldsymbol{\omega_{0}} + W^{1,\boldsymbol{p}}_{T}\left(\Omega;\boldsymbol{\Lambda^{k -1}}\right)$ satisfying 
$$\boldsymbol{d\beta} =  \boldsymbol{d\omega} \quad \text{ in } \Omega, $$ and a sequence 
$\lbrace \boldsymbol{\beta}^{s} \rbrace \subset \boldsymbol{\omega_{0}} + 
W^{1,\boldsymbol{p}}_{T}\left(\Omega;\boldsymbol{\Lambda^{k -1}}\right)$ such that 
\begin{gather*}
\boldsymbol{d\omega}^{s} =  \boldsymbol{d\beta}^{s} \quad \text{ in } \Omega,\text{ for every } s \\
\text{ and } \\ 
 \boldsymbol{\beta}^{s} \boldsymbol{\rightharpoonup} \boldsymbol{\beta} \text{ in } W^{d,\boldsymbol{p}}\left(\Omega;\boldsymbol{\Lambda^{k -1}}\right). 
\end{gather*}
Using theorem \ref{semicontinuity with x dependence}, we obtain,
\begin{align*}
 m=\underset{s\rightarrow\infty}{\lim\inf}\int_{\Omega}f\left( x,  \boldsymbol{d\omega}^{s}\right)  =\underset{s\rightarrow\infty}{\lim\inf}\int_{\Omega}f\left( x,
\boldsymbol{d\beta}^{s}\right)  &\geq\int_{\Omega}f\left( x, \boldsymbol{d\beta}\right)  \\&=\int_{\Omega
}f\left( x, \boldsymbol{d\omega}\right)  \geq m.
\end{align*}
This concludes the proof of the theorem.
\end{proof}
\begin{remark}
 It is easy to see that $\boldsymbol{\beta}$ in the proof of theorem \ref{Thm existence de minima ext vect} is a minimizer to the problem 
\[
\mathcal{(P}_{\boldsymbol{\delta},T}\mathcal{)}\quad \inf\left\{  \int_{\Omega} \left[ f\left( x, 
\boldsymbol{d\omega}\right) + \langle \boldsymbol{g} ; \boldsymbol{\omega} \rangle \right] :\boldsymbol{\omega} \in \boldsymbol{\omega_{0}}+W_{\delta, T}^{d,\boldsymbol{p}}
\left(  \Omega;\boldsymbol{\Lambda
^{k-1}}\right) \right\}  =m_{\delta,T},%
\]
under the hypotheses of the theorem \ref{Thm existence de minima ext vect} and thus $m_{\delta,T}=m.$ 
\end{remark}

\subsubsection{Existence theorem for polyconvex functions}
 \begin{theorem}\label{existence poly vect}
 Let $\Omega\subset\mathbb{R}^n$ be open, bounded, smooth and let $\boldsymbol{k}$ be given. Let $\boldsymbol{p} = (p_1,\ldots, p_m)$ where $1 < p_i < \infty$ 
 for all $1 \leq i \leq m $ be such that  
$\displaystyle  \sum_{i=1}^{m} \frac{\alpha_{i}}{p_{i}} < 1$ for any $\boldsymbol{\alpha}$ such that there exists $\boldsymbol{\xi} \in \boldsymbol{\Lambda^{k}}$ with 
$ \boldsymbol{\xi^{\alpha}} \neq 0.$ Let $F:\Omega\times\mathbb{R}^{\boldsymbol{\tau}(n, \boldsymbol{k})}\rightarrow \mathbb{R}\cup \lbrace +\infty \rbrace$ be a Carath\'eodory function, 
 satisfying for a.e $x \in \Omega,$ for every $\Xi \in \mathbb{R}^{\boldsymbol{\tau}(n, \boldsymbol{k})},$
\begin{align}\label{coercivitypoly}
  \Xi \mapsto &F(x, \Xi) \text{ is convex, } \notag \\
   &\text{and} \notag \\
  F(x, \Xi) &\geq a(x) + b \lVert \Xi_{1} \rVert^{\boldsymbol{p}}, 
 \end{align}
 where $\Xi = (\Xi_{1}, \ldots, \Xi_{N(\boldsymbol{k})}) \in \mathbb{R}^{\boldsymbol{\tau}(n, \boldsymbol{k})},$ $a \in L^{1}(\Omega),$ $b >0$ and 
 $$ \lVert \Xi_{1} \rVert^{\boldsymbol{p}} = \sum_{i=1}^{m} \lvert \Xi_{1}^{i}\rvert^{p_{i}}, \quad \text{ where } \Xi_{1} = (\Xi_{1}^{1},\ldots, \Xi_{1}^{m}) 
 \in \Lambda^{\boldsymbol{k}}.$$
 Let $\boldsymbol{g}
\in L^{\boldsymbol{p^{'}}}(\Omega; \boldsymbol{\Lambda^{k-1}})$ be such that
$\boldsymbol{\delta g} = 0$ in the sense of distributions and $\boldsymbol{\omega_{0}} \in W^{1,\boldsymbol{p}}\left(  \Omega;\boldsymbol{\Lambda^{k-1}}\right).$
Let \[
\mathcal{(P)}\ \inf\left\{   I(\boldsymbol{\omega}) =\int_{\Omega}\left[ F\left( x, 
T\left( \boldsymbol{d\omega} \right) \right) + \langle \boldsymbol{g} ; \boldsymbol{\omega} \rangle \right]  :\boldsymbol{\omega} 
\in \boldsymbol{\omega_{0}}+W_{0}^{1,\boldsymbol{p}}
\left(  \Omega;\boldsymbol{\Lambda^{k-1}}\right)  \right\}  =m.
\]
Then the problem $\mathcal{(P)}$ has a minimizer.
 \end{theorem}
\begin{proof}
 By the same argument as in the proof of theorem \ref{Thm existence de minima ext vect}, Step 1, we can assume that $\boldsymbol{g} = 0.$ 
 Let $\lbrace \boldsymbol{\omega}^{s} \rbrace$ be a minimizing sequence of 
$\mathcal{(P)}$. By \eqref{coercivitypoly}, there exists a constant $c>0$ such that 
$$\lVert \boldsymbol{d\omega}^{s}\rVert_{L^{\boldsymbol{p}}\left(\Omega;\boldsymbol{\Lambda^{k}}\right)} \leq c .$$ Thus we have
$$  \boldsymbol{d\omega}^{s} \boldsymbol{\rightharpoonup} \boldsymbol{\zeta} \quad \text{ in } L^{\boldsymbol{p}}\left(\Omega;\boldsymbol{\Lambda^{k}}\right) . $$ By the weak convergence, 
it also follows that $d\boldsymbol{\zeta} = 0$ in the sense of distributions and $\nu\wedge\boldsymbol{\zeta} = \nu\wedge \boldsymbol{d\omega_{0}}$ on $\partial\Omega.$
Thus, we can find $\boldsymbol{\omega}\in \boldsymbol{\omega_{0}} + W^{1,\boldsymbol{p}}_{0}\left(\Omega;\boldsymbol{\Lambda^{k -1}}\right)$ such that 
\begin{equation*}
   \left\lbrace \begin{aligned}
                \boldsymbol{d\omega} = \boldsymbol{\zeta}  \qquad &\text{ in } \Omega, \\
                \boldsymbol{\omega} = \boldsymbol{\omega_{0}} \qquad &\text{  on } \partial\Omega.
                \end{aligned} 
                \right. 
                \end{equation*}
Thus, we have,                 
$$  \boldsymbol{d\omega}^{s} \boldsymbol{\rightharpoonup} \boldsymbol{d\omega} \quad \text{ in } L^{\boldsymbol{p}}\left(\Omega;\boldsymbol{\Lambda^{k}}\right) . $$
Then by the assumption on $\boldsymbol{p}$, theorem \ref{weakcontgeneral} implies,   
\begin{equation}\label{L1convergence}
 T\left( \boldsymbol{d\omega}^{s} \right) \boldsymbol{\rightharpoonup} T \left( \boldsymbol{d\omega} \right) \quad \text{ in } L^{1}\left(\Omega;\mathbb{R}^{\boldsymbol{\tau}(n, \boldsymbol{k})} \right) .
\end{equation}

Since $\Xi \mapsto F(x, \Xi)$ is convex, we obtain $I(\boldsymbol{\omega}) = m.$ 
\end{proof}
\begin{remark}
 The pointwise coercivity condition \eqref{coercivitypoly} used here can be unnecessarily strong in practice for applications. Indeed, any condition that 
 ensures the convergence \eqref{L1convergence} for all minimizing sequences is enough, as the proof shows. As an example, the `mean coercivity' condition  introduced in 
 Iwaniec-Lutoborski (\cite{iwaniec-lutoborski-null-lagrangian}, definition 9.1) works as well.    
\end{remark}

\smallskip

\noindent\textbf{Acknowledgement.} The author thanks Bernard Dacorogna, Saugata Bandyopadhyay and Jan Kristensen for helpful 
comments and discussions. Also, some of the results in this work constitutes a part of author's doctoral thesis in EPFL, whose support and facilities are gratefully acknowledged.

\end{document}